\documentclass[12pt]{article}
\usepackage{amsmath,amsfonts,amssymb,amsthm}

\newtheorem{theo}{Theorem}
\newtheorem*{coro}{Corollary}

\theoremstyle{remark}

\numberwithin{equation}{section}
\newcommand\la{\lambda}
\newcommand\ta{\sigma}
\newcommand\al{\alpha}
\newcommand\om{\omega}
\newcommand\ep{\varepsilon}

\title{An (inverse) Pieri formula\\ for Macdonald polynomials of type $C$}
\author{Michel Lassalle\\
\small Centre National de la Recherche Scientifique\\[-0.8ex]
\small Institut Gaspard-Monge, Universit\'e de Marne-la-Vall\'ee\\[-0.8ex]
\small 77454 Marne-la-Vall\'ee Cedex, France\\[-0.8ex]
\small \texttt{lassalle@univ-mlv.fr}\\[-0.8ex]
\small \texttt{http://igm.univ-mlv.fr/{\textasciitilde}lassalle}
}
\date{
}

\begin{document}
\maketitle
\begin{abstract}
We give an explicit Pieri formula for Macdonald polynomials attached to the root system $C_n$ (with equal multiplicities). By inversion we obtain an explicit expansion for two-row Macdonald polynomials of type $C$.
\end{abstract}

\section{Introduction}

In the eighties, I. G. Macdonald introduced a new family of orthogonal polynomials, which are Laurent polynomials in several variables, and generalize the Weyl characters of compact simple Lie groups~\cite{Ma3,Ma4}. In the most simple situation, a family 
$P_{\la}^{(R)}(q,t)$ of such polynomials, depending rationally 
on two parameters $q,t$, is attached to any reduced irreducible root system $R$. 

These orthogonal polynomials are elements of the group algebra of 
the weight lattice of $R$, invariant under the action of the Weyl 
group. They are indexed by the set $P^+$ of dominant weights. For $t=q$, they  correspond to the Weyl characters 
$\chi_{\la}^{(R)}$ of compact simple Lie groups. 

When $R$ is of type $A$, the orthogonal polynomials $P_{\la}^{(A)}(q,t)$ correspond to the symmetric functions $P_{\la}(q,t)$ previously studied in \cite{Ma1,Ma2}.
 
Let $\la$ be a dominant weight of $R$, $r$ some positive integer and $\omega_s$ some fundamental weight. By analogy with the type $A$ case, the product decomposition
\[P_{r \omega_s}^{(R)}\, P_{\la}^{(R)}=\sum_{\begin{subarray}{c}\tau\in\Sigma \\
\la+\tau\in P^+\end{subarray}} C_\tau \, P_{\la+\tau}^{(R)}\]
is called a ``Pieri formula'' for the Macdonald polynomials of type $R$. 

When $R=A_n$ the range $\Sigma$ and the coefficients $C_\tau$ are explicitly known for $r=1$, $s$ arbitrary and $s=1$, $r$ arbitrary~\cite[(6.24), p. 340]{Ma2}. Moreover there is a duality property connecting these two cases.

The situation is very different when $R$ is not of type $A$. General results~\cite{Ma5} entail that $\Sigma$ is formed by the ``integral points'' of the convex hull of the Weyl group orbit of $r\omega_s$. But no formula is known for the coefficients $C_\tau$, except when $\tau$ belongs to the boundary of this convex hull. Finding a general explicit expression of $C_\tau$ is still a very difficult open problem.

In this paper we perform some computation of $C_\tau$ when $R=C_n$. For this root system we give the explicit decomposition of the product $P_{r\om_1}P_{s\om_1}$. Remarkably the coefficients $C_\tau$ appear to be fully factorized (which is not a general property). Moreover this Pieri formula is easily inverted, which yields the explicit expansion of any $P_{\la_1\om_1+\la_2\om_2}$ in terms of products $P_{r\om_1}P_{s\om_1}$. 

For $n=2$ this expansion completely determines the Macdonald polynomials attached to $C_2$ and its dual root system $B_2$.

However our results are obtained under the technical assumption that the same parameter $t$ is associated with short and long roots. The case of distinct parameters is yet unknown and appears to be much more intricate.

The paper is organized as follows. Sections 2 and 3 are devoted to general facts about Macdonald polynomials, including the Pieri formula for a (quasi) minuscule weight. In Section 4 these results are specified for the root system $C_n$. Our Pieri formula is presented in Section 5, proved in Section 6 and inverted in Section 7. Technical material is given in Sections 8-10, including $\la$-ring calculus and a very remarkable rational identity. 

This multivariate identity presents an interest by itself. Section 12 is devoted to some basic $q$-hypergeometric identities obtained by its ``(multiple) principal specialization'', and outlines its links with previous results.

\smallskip
{\it Acknowledgements.} Thanks are due to the anonymous referees for helpful comments. I am much indebted to Michael Schlosser for advice and discussions, in particular for making me aware of~\cite{Ro}.

\section{Macdonald polynomials}

In this section we introduce our notations, and recall some 
general facts about Macdonald polynomials. For more details the reader is referred to \cite{Ma3,Ma4,Ma5}. 

The most general class of Macdonald polynomials is associated with a 
pair of root systems $(R,S)$, spanning the same vector space and 
having the same Weyl group, with $S$ reduced. Throughout this paper, we shall only consider the case $S=R$.

Let $V$ be a finite-dimensional real vector space endowed with a
positive definite symmetric bilinear form $\langle  u, v 
\rangle$. For all $v\in V$, we write $\vert v\vert = \langle v, v \rangle ^{1/2}$ and $v^\lor =2v/\vert v\vert^2$.

Let $R\subset V$ be a reduced irreducible root system, $W$ the Weyl group of $R$, $R^+$ the set of positive roots, $\{\alpha_1,\dots,\alpha_n\}$ 
the basis of simple roots, and $R^\lor  =\{ \alpha^\lor \mid \alpha\in R\}$ the dual root system.

The fundamental weights $\om_i$ 
are defined by $\langle\om_{i}, \al_{j}^\lor \rangle=\delta _{ij}$. Let 
\[Q=\sum_{i=1}^n\mathbb{Z}\ \alpha_i, \quad\quad Q^+=\sum_{i=1}^n\mathbb{N}\ \alpha_i\]
be the root lattice of $R$ and its positive octant, and 
\[P=\sum_{i=1}^n\mathbb{Z}\ \om_i, \quad\quad P^+=\sum_{i=1}^n\mathbb{N}\ \om_i\]
the weight lattice of $R$ and the cone of dominant weights. A partial order is defined on $P$ by
$\la\ge\mu$ if and only $\la-\mu\in Q^+$.

Let $A$ denote the group algebra over $\mathbb{R}$ of the free Abelian group $P$. For each $\la\in P$ let $e^\la$ denote the corresponding
element of $A$, subject to the multiplication rule 
$e^\la e^\mu = e^{\la +\mu}$. The set $\{e^\la, \la\in P\}$ is an $\mathbb{R}$-basis of $A$.

The Weyl group $W$ acts on $P$, hence on $A$ by $w(e^\la)=e^{w\la}$.
Let $A^{W}$ denote the subspace of $W$-invariants in $A$. 
Such elements are called ``symmetric polynomials''.
There are two important examples of a basis of $A^{W}$, both indexed by dominant weights $\la\in P^{+}$. 

The 
first one is given by the orbit-sums
\[m_\la = \sum_{\mu\in W\la}e^{\mu}.\]
The second one is provided by the Weyl characters defined as 
follows. Let
\[\delta =\prod_{\al\in R^+}
(e^{\al/2}-e^{-\al/2})=e^{-\rho} \prod_{\al\in
R^+} (e^{\al}-1),\]
with $\rho=\frac{1}{2} \sum_{\al \in R^{+}} \al \in P^+$.
Then $w\delta = \ep(w)\delta$ for any $w\in W$, where
$\ep(w)=\det(w)=\pm 1$. For all $\la\in P$ the element
\[\chi_\la = \delta^{-1}\sum_{w\in W}\ep(w)e^{w(\la +
\rho)}\]
is in $A^{W}$, and the set $\{\chi_\la,\ \la\in P^{+}\}$ 
forms an $\mathbb{R}$-basis of $A^{W}$.

Let $0< q <1$. For any real number $k\in\mathbb{R}$, the classical $q$-shifted factorial $(u;q)_k$ is defined by
\begin{equation*}
(u;q)_{\infty}=\prod_{j\ge 0}(1-uq^j),\qquad
(u;q)_k=(u;q)_{\infty}/(uq^k;q)_{\infty}.
\end{equation*}
For each $\al\in R$ let $t_\al=q^{k_\al}$ be a positive real 
number such that $t_\al=t_\beta$ if $\vert\alpha\vert=\vert\beta\vert$. 
There are at most two different values for the $t_\alpha$'s, depending on whether $\al$ is a short or a long root. We define
\[\rho_k=\frac{1}{2} \sum_{\alpha \in R^{+}} k_\al \al,
\quad \quad \rho_k^\lor=\frac{1}{2} \sum_{\alpha \in R^{+}} k_\al \al^\lor.\]
Observe that $\rho_k^\lor$ is not $2\rho_k/\vert \rho_k\vert^2$.

If $f=\sum_{\la \in P} \ a_{\la}\, e^{\la} \in A$, let
$\overline{f}=\sum_{\la \in P} \ a_{\la}\,e^{-\la }$ and
$[\,f\,]_{1}$ its constant term $a_{0}$.
The inner product defined on $A$ by
\[\langle f, g\rangle _{q,t}=\frac{1}{|W|}[f\bar g\Delta]_{1},\] 
with $|W|$ the order of $W$ and
\[\Delta =\prod_{{\alpha\in R}}\frac 
{(e^\alpha;q)_\infty}{(t_\alpha e ^\alpha;q)_\infty}=
\prod_{{\alpha\in R}}(e^\alpha;q)_{k_\al}\]
is non degenerate and $W$-invariant.

There exists a unique basis $\{P_\la, \ \la \in P^{+}\}$ of $A^W$, 
called Macdonald polynomials, such that
\begin{itemize}
\item[(i)] $P_\la = m_\la +
\sum_{\mu\in P^+,\ \mu < \la} \ a_{\la\mu}(q,t)\ m_\mu,$ 

where the coefficients $a_{\la\mu}(q,t)$ are rational functions of 
$q$ and the $t_\alpha$'s,
\item[(ii)] $\langle P_\la, P_\mu\rangle  _{q,t}= 0$ if $\la\neq\mu$.
\end{itemize}

It is clear that the $P_\la$'s, if they exist, are unique. Their 
existence is proved as eigenvectors of an operator $A^W\rightarrow A^W$, 
self-adjoint with respect to $\langle \ ,\ \rangle _{q,t}$ 
and having its eigenvalues all distinct. This operator may be constructed as follows \cite{Ma3,Ma4}.

A weight $\pi$ of $R^\lor$ is called minuscule if $\langle \pi,\alpha \rangle \in \{0,1\}$ for $\alpha \in R^+$, and quasi-minuscule if $\langle \pi,\alpha \rangle \in \{0,1,2\}$ for $\alpha \in R^+$. Let $T_\pi$ be the translation operator defined on $A$ by
$T_{\pi}(e^{\la})=q^{\langle\pi ,\la \rangle}e^{\la}$ for any $\la \in P$, and
\[\Phi_\pi =T_\pi(\Delta_+)/\Delta_+
\quad \quad \text{with} \quad \quad
\Delta_+=\prod_{\al\in R^+} (e^\alpha;q)_{k_\al}.\]
Two situations may be considered.

\textit{First case : $\pi$ is a minuscule weight of $R^\lor$.}

Such a weight only exists when $R$ is of 
type $B,C,D$ or $E_6,E_7$. It is necessarily a fundamental weight of $R^\lor$. 
Let $E_\pi$ be the self-adjoint operator defined by
\[E_\pi f = \sum_{w\in W} w(\Phi_\pi\, T_\pi f),\]
where $\Phi_\pi$ is now given by
\[\Phi_\pi=\prod_{\begin{subarray}{c}\al\in R^+ \\
\langle\pi,\al\rangle=1\end{subarray}}
\frac {1-t_\al e^\al} {1-e^\al}=
\prod_{\al\in R^+}\frac {1-t_\al^{\langle\pi,\al\rangle} e^{\al}} 
{1-e^{\al}}.\]
Macdonald polynomials $P_\la$ are defined as eigenvectors 
of $E_\pi$, namely
\[E_\pi\, P_\la=c_{\la}\, P_\la
\quad \quad \text{with} \quad \quad
c_{\la}=q^{\langle\pi,\rho_k\rangle}\sum_{w\in W}q
^{\langle w\pi,\la + \rho_k\rangle}.\]

\textit{Second case : $\pi$ is a quasi-minuscule weight of $R^\lor$.}

When $R$ is $E_8$, $F_4$ or $G_2$, $R^\lor$ has no minuscule weight. However a quasi-minuscule weight always exists, given by $\pi=\varphi^\lor$, where $\varphi$ is the highest root of $R$. For the types just mentioned, it is the only one.

In other words, the family $\{\al\in R^+ :
\langle\pi,\al\rangle=2\}$ is either empty (if $\pi$ is minuscule) or reduced to the single element $\varphi$ (if $\pi=\varphi^\lor$).

For $\pi$ a quasi-minuscule weight, let $F_\pi$ be the self-adjoint operator defined by
\[F_\pi f = \sum_{w\in W} w(\Phi_\pi\, (T_\pi-1) f),\]
where $\Phi_\pi$ is now given by
\begin{equation*}
\begin{split}
\Phi_\pi&=\prod_{\begin{subarray}{c}\al\in R^+ \\
\langle\pi,\al\rangle=1\end{subarray}}
\frac {1-t_\al e^\al} {1-e^\al}\,
\prod_{\begin{subarray}{c}\al\in R^+ \\
\langle\pi,\al\rangle=2\end{subarray}}
\frac {1-t_\al e^\al} {1-e^\al}
\frac {1-qt_\al e^\al} {1-q e^\al}\\
&=\prod_{\al\in R^+}\frac {1-t_\al^{\langle\pi,\al\rangle} e^{\al}} 
{1-e^{\al}}\,
\prod_{\begin{subarray}{c}\al\in R^+ \\
\langle\pi,\al\rangle=2\end{subarray}}
\frac {1-t_\al e^\al} {1-t_\al^2e^\al}
\frac {1-qt_\al e^\al} {1-q e^\al}.
\end{split}
\end{equation*}
Then Macdonald polynomials $P_\la$ are eigenvectors 
of $F_\pi$. We have
\[F_\pi\, P_\la=c_{\la}^\prime\, P_\la
\quad \quad \text{with} \quad \quad
c_{\la}^\prime=q^{\langle\pi,\rho_k\rangle}\sum_{w\in W}\left(
q^{\langle w\pi,\la + \rho_k\rangle}-q^{\langle w\pi, \rho_k\rangle}\right).\]

If $\pi$ is minuscule, the definitions of $E_\pi$ and $F_\pi$ are equivalent because we have
\[\sum_{w\in W} w\Phi_\pi=\sum_{w\in W}
\prod_{\begin{subarray}{c}\al\in R^+ \\
\langle\pi,\al\rangle=1\end{subarray}}
\frac {1-t_\al e^{w\al}} {1-e^{w\al}}=
q^{\langle\pi,\rho_k\rangle}\sum_{w\in W}\,q^{\langle w\pi, \rho_k\rangle},\]
which is a consequence of the Macdonald identity
\[\sum_{w\in W}
\prod_{\al\in R^+}
\frac {1-u_\al e^{-w\al}} {1-e^{-w\al}}=
\sum_{w\in W}
\prod_{\al\in R^+\cap -wR^+} u_\al,\]
proved in~\cite[Theorem 2.8]{Ma6} for any family of indeterminates $\{u_\al, \al\in R^+\}$.

We may regard any $f=\sum_{\la \in P}  f_\lambda\, e ^\lambda \in A$,
with only finitely many nonzero coefficients, as a function on $V$ by putting 
for any $x\in V$,
\[f(x)=\sum_{\la \in P} f_\lambda\, q ^{\langle\lambda,x\rangle}.\]
With this convention we obviously have
\[(T_\tau f)(x)=f(x+\tau).\]
Then Macdonald polynomials satisfy the two following 
properties proved by Cherednik~\cite{C}:
\begin{itemize}
\item[(i)] Specialization. For any $\la\in P^+$ we have
\[P_\la(\rho_k^\lor)= q^{-\langle\la,\rho_k^\lor\rangle} \,
\prod_{\al \in R^+} 
\frac{(q^{\langle\rho_k,\al^\lor\rangle}t_{\al};q)_{\langle\la,\al^\lor\rangle}}
{(q^{\langle\rho_k,\al^\lor\rangle};q)_{\langle\la,\al^\lor\rangle}}.\]

\item[(ii)] Symmetry. Let $P^\lor$ be the weight lattice of $R^\lor$, and for any $\mu \in (P^\lor)^+$ let $P_\mu$ be the associated Macdonald polynomial. Define 
$$\tilde{P}_\la=P_\la/P_\la(\rho_k^\lor),\qquad \tilde{P}_\mu=P_\mu/P_\mu(\rho_k).$$ 
Then we have
\[\tilde{P}_\la(\mu+\rho_k^\lor)=\tilde{P}_\mu(\la+\rho_k).\]
\end{itemize}

\section{The Pieri formula for a (quasi) minuscule weight}

For any vector $\tau\in V$ define
\[\Sigma(\tau)= C(\tau) \cap (\tau+Q)=\bigcap_{w\in W} w(\tau-Q^+)\]
with $C(\tau)$ the convex hull of the Weyl group orbit $W\tau$.

Let $\la\in P^+$ and $\om$ be a fundamental weight. We consider the Pieri formula
\[P_{\om}\, P_{\la}=\sum_{\begin{subarray}{c}\tau\in\Sigma \\
\la+\tau\in P^+\end{subarray}} C_\tau \, P_{\la+\tau}.\]
By general results~\cite[(5.3.8), p. 104]{Ma5}, it is known that
the range $\Sigma$ on the right-hand side is equal to $\Sigma(\om)$. But an explicit formula for the coefficients $C_\tau$ is yet unknown.

However when $\om$ is a minuscule or quasi-minuscule weight of $R$, an explicit expression for $C_\tau$ can be derived from the definition of Macdonald polynomials attached to the dual root system $R^\lor$. Although this duality method is known to experts (see~\cite[Appendix]{D1} or~\cite[Section 4]{D2}), we think useful to enter into details.

Observe that if $\om$ is a minuscule weight of $R$ we have $\Sigma(\om)=W\om$ and $P_{\om}=m_{\om}$. 

\begin{theo}
Let $\om\in P^+$ be a minuscule (fundamental) weight of $R$. Then we have
\[P_{\om}\, P_{\la}=\sum_{\begin{subarray}{c}\tau\in W\om \\
\la+\tau\in P^+\end{subarray}} C_\tau \, P_{\la+\tau},\]
with
\[C_\tau= \prod_{\begin{subarray}{c}\al \in R^+\\
\langle\tau,\al^\lor\rangle=-1 \end{subarray}}
\frac{1-q^{\langle\la+\rho_k,\al^\lor\rangle}t_\al^{-1}}
{1-q^{\langle\la+\rho_k,\al^\lor\rangle}}
\frac{1-q^{\langle\la+\rho_k,\al^\lor\rangle-1}t_\al}
{1-q^{\langle\la+\rho_k,\al^\lor\rangle-1}}.\]
Equivalently
\[P_{\om}\, \tilde{P}_{\la}=\sum_{\begin{subarray}{c}\tau\in W\om \\
\la+\tau\in P^+\end{subarray}} C_\tau \, \tilde{P}_{\la+\tau},\]
with
\[C_\tau=q^{-\langle\tau,\rho^\lor_k\rangle}\,
\prod_{\al \in R^+}
\frac{1-t_\al^{\langle\tau,\al^\lor\rangle} q^{\langle\la+\rho_k,\al^\lor\rangle}}
{1-q^{\langle\la+\rho_k,\al^\lor\rangle}}.\]
\end{theo}
\begin{proof}
The equivalence of both formulations is a consequence of the specialization formula, together with the fact that for $\al \in R^+$ we have $\langle \tau,\al^\lor\rangle\in\{-1,0,1\}$, since $\langle\om,\al^\lor\rangle\in\{0,1\}$ and $W$ permutes roots. 

It is equivalent to prove the second formulation for the dual root system $R^\lor$. Then it writes as
\[m_\pi\, \tilde{P}_\mu=\sum_{\begin{subarray}{c}\tau\in W\pi \\
\mu+\tau \in (P^\lor)^+\end{subarray}} C_\tau \, \tilde{P}_{\mu+\tau},\]
with $\pi$ a minuscule weight of $R^\lor$, $\mu \in (P^\lor)^+$ and
\[C_\tau=q^{-\langle\tau,\rho_k\rangle}\,
\prod_{\al \in R^+}
\frac{1-t_\al^{\langle\tau,\al\rangle} q^{\langle\mu+\rho^\lor_k,\al\rangle}}
{1-q^{\langle\mu+\rho^\lor_k,\al\rangle}}.\]

Let $\kappa \in P^+$ arbitrary. The associated Macdonald polynomial may be defined by $E_\pi\, P_\kappa=c_{\kappa}\, P_\kappa$ with
\[
c_{\kappa}=q^{\langle\pi,\rho_k\rangle}\sum_{w\in W}q
^{\langle w\pi,\kappa + \rho_k\rangle}=q^{\langle\pi,\rho_k\rangle} |W_\pi| \,m_\pi(\kappa+\rho_k),\]
and $W_\pi$ the stabilizer of $\pi$ in $W$. On the other hand, $E_\pi$ is given by
\[E_\pi = \sum_{w\in W} w(\Phi_\pi\, T_\pi)=
\sum_{w\in W} \prod_{\al\in R^+}\frac {1-t_\al^{\langle\pi,\al\rangle} e^{w\al}} {1-e^{w\al}} T_{w\pi}.\]
This can be written as
\[E_\pi = q^{\langle\pi,\rho_k\rangle} |W_\pi| \,
\sum_{\tau\in W\pi} q^{-\langle\tau,\rho_k\rangle} 
\prod_{\al\in R^+}\frac {1-t_\al^{\langle\tau,\al\rangle} e^{\al}} 
{1-e^{\al}} T_{\tau},\]
because for any $w \in W$ we have
\[\prod_{\al\in R^+}\frac {1-t_\al^{\langle\pi,\al\rangle} e^{w\al}} {1-e^{w\al}}=\prod_{\al\in R^+\cap wR^+}\frac {1-t_\al^{\langle w\pi,\al\rangle} e^{\al}} 
{1-e^{\al}}\,
\prod_{\al\in R^+\cap -wR^+} t_\al^{-\langle w\pi,\al\rangle}
\frac {1-t_\al^{\langle w\pi,\al\rangle} e^{\al}} 
{1-e^{\al}},\]
and
\[q^{\langle\pi-w\pi,\rho_k\rangle} =
\prod_{\al\in R^+\cap -wR^+} t_\al^{-\langle w\pi,\al\rangle}.\]
Thus for $x\in V$ the definition of $P_\kappa$ writes as
\[m_\pi(\kappa+\rho_k)\,P_\kappa(x)=
\sum_{\tau\in W\pi} q^{-\langle\tau,\rho_k\rangle} 
\prod_{\al\in R^+}\frac {1-t_\al^{\langle\tau,\al\rangle} q^{\langle\al,x\rangle}} 
{1-q^{\langle\al,x\rangle}} P_\kappa(x+\tau).\]
Choosing $x=\mu+\rho^\lor_k$, dividing both sides by $P_\kappa(\rho^\lor_k)$ and making use of the symmetry property, we can conclude.
\end{proof}

There is an analogous result when $\om$ is a quasi-minuscule weight. Then we have $\Sigma(\om)=W\om \cup \{0\}$ and $P_{\om}=m_{\om}+ \mathrm{constant}$.
\begin{theo}
Let $\om\in P^+$ be a quasi-minuscule weight of $R$. Then we have
\[\big(P_{\om}-P_{\om}(\rho_k^\lor)\big)\, P_{\la}=\sum_{\begin{subarray}{c}\tau\in W\om \\
\la+\tau\in P^+\end{subarray}} \big(C_\tau P_{\la+\tau} - D_\tau P_\la\big),\]
with 
\begin{equation*}
\begin{split}
C_\tau&= \prod_{\begin{subarray}{c}\al \in R^+\\
\langle\tau,\al^\lor\rangle=-1 \end{subarray}}
\frac{1-q^{\langle\la+\rho_k,\al^\lor\rangle}t_\al^{-1}}
{1-q^{\langle\la+\rho_k,\al^\lor\rangle}}
\frac{1-q^{\langle\la+\rho_k,\al^\lor\rangle-1}t_\al}
{1-q^{\langle\la+\rho_k,\al^\lor\rangle-1}}\\
&\times \prod_{\begin{subarray}{c}\al \in R^+\\
\langle\tau,\al^\lor\rangle=-2 \end{subarray}}
\frac{1-q^{\langle\la+\rho_k,\al^\lor\rangle}t_\al^{-1}}
{1-q^{\langle\la+\rho_k,\al^\lor\rangle}}
\frac{1-q^{\langle\la+\rho_k,\al^\lor\rangle-1}t_\al^{-1}}
{1-q^{\langle\la+\rho_k,\al^\lor\rangle-1}}\,
\frac{1-q^{\langle\la+\rho_k,\al^\lor\rangle-1}t_\al}
{1-q^{\langle\la+\rho_k,\al^\lor\rangle-1}}
\frac{1-q^{\langle\la+\rho_k,\al^\lor\rangle-2}t_\al}
{1-q^{\langle\la+\rho_k,\al^\lor\rangle-2}},\\
D_\tau&= q^{-\langle\tau,\rho_k^\lor\rangle} 
\prod_{\begin{subarray}{c}\al \in R^+\\
\langle\tau,\al^\lor\rangle=\pm 1 \end{subarray}}
\frac{1-q^{\langle\la+\rho_k,\al^\lor\rangle}t_\al^{\pm 1}}
{1-q^{\langle\la+\rho_k,\al^\lor\rangle}}\,
\prod_{\begin{subarray}{c}\al \in R^+\\
\langle\tau,\al^\lor\rangle=\pm 2 \end{subarray}}
\frac{1-q^{\langle\la+\rho_k,\al^\lor\rangle}t_\al^{\pm 1}}
{1-q^{\langle\la+\rho_k,\al^\lor\rangle}}
\frac{1-q^{\langle\la+\rho_k,\al^\lor\rangle\pm 1}t_\al^{\pm 1}}
{1-q^{\langle\la+\rho_k,\al^\lor\rangle\pm 1}}.
\end{split}
\end{equation*}
Equivalently
\[\big(P_\om-P_\om(\rho_k^\lor)\big)\,\tilde{P}_\la=
\sum_{\tau\in W\om} q^{-\langle\tau,\rho_k^\lor\rangle} 
\prod_{\al\in R^+}\frac{(q^{\langle\la+\rho_k,\al^\lor\rangle}t_\al^{\epsilon_\al};q^{\epsilon_\al})_{|\langle\tau,\al^\lor\rangle|}}
{(q^{\langle\la+\rho_k,\al^\lor\rangle};q^{\epsilon_\al})_{|\langle\tau,\al^\lor\rangle|}}
\big(\tilde{P}_{\la+\tau}-\tilde{P}_\la\big),\]
with $\epsilon_\al$ the sign of $\langle\tau,\al^\lor\rangle$.
\end{theo}
\begin{proof}
The equivalence of both formulations results from the specialization formula
\[q^{-\langle\tau,\rho_k^\lor\rangle}
\frac{P_\la(\rho_k^\lor)}{P_{\la+\tau}(\rho_k^\lor)}=
\prod_{\al \in R^+} 
\frac{(q^{\langle\la+\rho_k,\al^\lor\rangle};q)_{\langle\tau,\al^\lor\rangle}}
{(q^{\langle\la+\rho_k,\al^\lor\rangle}t_{\al};q)_{\langle\tau,\al^\lor\rangle}},\]
together with the fact that for $\al \in R^+$ we have $\langle \tau,\al^\lor\rangle\in\{-2,-1,0,1,2\}$, since $W$ permutes roots.

It is enough to prove the second formulation for the dual root system $R^\lor$. Let $\pi$ a quasi-minuscule weight of $R^\lor$ and $\kappa \in P^+$ arbitrary. The Macdonald polynomial $P_\kappa$ may be defined by $F_\pi\, P_\kappa=c_{\kappa}^\prime\, P_\kappa$ with
\[
c_{\kappa}^\prime=q^{\langle\pi,\rho_k\rangle}\sum_{w\in W}\left(
q^{\langle w\pi,\kappa + \rho_k\rangle}-q^{\langle w\pi, \rho_k\rangle}\right)=q^{\langle\pi,\rho_k\rangle} |W_\pi| \,(m_\pi(\kappa+\rho_k)-m_\pi(\rho_k)).\]
But $F_\pi$ is given by
\[F_\pi = \sum_{w\in W} w(\Phi_\pi\, (T_\pi-1))=
\sum_{w\in W} 
\prod_{\al\in R^+}\frac {1-t_\al^{\langle\pi,\al\rangle} e^{w\al}} 
{1-e^{w\al}}\,
\prod_{\begin{subarray}{c}\al\in R^+ \\
\langle\pi,\al\rangle=2\end{subarray}}
\frac {1-t_\al e^{w\al}} {1-t_\al^2e^{w\al}}
\frac {1-qt_\al e^{w\al}} {1-q e^{w\al}} (T_{w\pi}-1).\]
As in the proof of Theorem 1, the first product at the right-hand side is
\[\prod_{\al\in R^+}\frac {1-t_\al^{\langle\pi,\al\rangle} e^{w\al}} 
{1-e^{w\al}}=q^{\langle\pi-w\pi,\rho_k\rangle}\,
\prod_{\al\in R^+}\frac {1-t_\al^{\langle w\pi,\al\rangle} e^{\al}} 
{1-e^{\al}}.\]
The second product may be written
\[\prod_{\begin{subarray}{c}\al\in R^+ \\
\langle\pi,\al\rangle=2\end{subarray}}
\frac {1-t_\al e^{w\al}} {1-t_\al^2e^{w\al}}
\frac {1-qt_\al e^{w\al}} {1-q e^{w\al}}=\prod_{\begin{subarray}{c}\gamma\in R^+ \\
\langle w\pi,\gamma\rangle=\pm 2\end{subarray}}
\frac {1-t_\gamma^{\pm 1} e^{\gamma}} 
{1-t_\gamma^{\pm 2} e^{\gamma}}
\, \frac{1-q^{\pm 1} t_\gamma^{\pm 1} e^{\gamma}}
{1-q^{\pm 1} e^{\gamma}}.\]
Finally, writing $\epsilon_\al$ for the sign of $\langle\tau,\al\rangle$, we have
\[F_\pi = q^{\langle\pi,\rho_k\rangle}\,|W_\pi|\,\sum_{\tau\in W\pi} 
q^{-\langle\tau,\rho_k\rangle}\, \prod_{\al\in R^+}
\frac{(t_\al^{\epsilon_\al}e^\al;q^{\epsilon_\al})_{|\langle\tau,\al\rangle|}}{(e^\al;q^{\epsilon_\al})_{|\langle\tau,\al\rangle|}} \, (T_{\tau}-1).\]
The definition of $P_\kappa$ yields
\[(m_\pi(\kappa+\rho_k)-m_\pi(\rho_k))\,P_\kappa(x)=
\sum_{\tau\in W\pi} q^{-\langle\tau,\rho_k\rangle} 
\prod_{\al\in R^+}
\frac{(q^{\langle\al,x\rangle}t_\al^{\epsilon_\al};q^{\epsilon_\al})_{|\langle\tau,\al\rangle|}}{(q^{\langle\al,x\rangle};q^{\epsilon_\al})_{|\langle\tau,\al\rangle|}}\, (P_\kappa(x+\tau)-P_\kappa(x)).\]
Choosing $x=\mu+\rho^\lor_k$, with $\mu \in (P^\lor)^+$, dividing both sides by $P_\kappa(\rho^\lor_k)$ and making use of the symmetry property, we get
\[(m_\pi-m_\pi(\rho_k))\,P_\mu=
\sum_{\tau\in W\pi} q^{-\langle\tau,\rho_k\rangle} 
\prod_{\al\in R^+}
\frac{(q^{\langle\al,\mu+\rho^\lor_k\rangle}t_\al^{\epsilon_\al};q^{\epsilon_\al})_{|\langle\tau,\al\rangle|}}
{(q^{\langle\al,\mu+\rho^\lor_k\rangle};q^{\epsilon_\al})_{|\langle\tau,\al\rangle|}}
\left(\frac{P_\mu(\rho_k)}{P_{\mu+\tau}(\rho_k)}P_{\mu+\tau}-P_\mu\right).\]
We conclude by dividing both sides by $P_\mu(\rho_k)$.
\end{proof}

\section{The root system $C_n$}

From now on we assume that $R=C_n$. We identify $V$ with $\mathbb{R}^n$ with the standard basis $\ep_1,\ldots,\ep_n$. Defining $x_i=e^{\ep_i},\ 1\le i\le n$, we regard Macdonald polynomials as Laurent polynomials of $n$ variables $x_1,\ldots,x_n$.

The set of positive roots is the union of short roots
$R_{1}=\{\ep_{i}\pm \ep_{j}, 1\le i < j\le n\}$ and long roots
$R_{2}=\{2\ep_{i}, 1\le i\le n\}$. Throughout this paper we shall assume that the same parameter $t=q^k$ is associated with short and long roots. We have $\rho_k=k\sum_{i=1}^n (n-i+1) \ep_i$ and $\rho_k^\lor= k\sum_{i=1}^n (n-i+\frac{1}{2}) \ep_i.$

The Weyl group $W$ is the semi-direct product of the 
permutation group $S_{n}$ by $(\mathbb{Z}/2\mathbb{Z})^n$. 
It acts on $V$ by signed permutation of components. 
The fundamental weights are given by 
$\om_{i}=\sum_{j=1}^i\ep_j,\ 1\le i\le n$. The dominant 
weights $\la\in P^+$ can be identified with vectors 
$\la=\sum_{i=1}^n \la_i \ep_i$ such that
$(\la_1,\la_2,\ldots,\la_n)$ is a partition. 

The dual root system $B_n$ has one minuscule 
weight $\pi=\frac{1}{2}(\ep_1+\ldots+\ep_n)$. Its $W$-orbit is formed by vectors
$\frac{1}{2}(\sigma_1\ep_1+\ldots+\sigma_n\ep_n)$ with 
$\sigma\in(-1,+1)^n$. We have
\begin{equation}
\Phi_\pi =
\prod_{i=1}^n\frac {1-tx_i^2} {1-x_i^2} \
\prod_{1\le i < j \le n}\frac {1-tx_ix_j} {1-x_ix_j},
\end{equation} 
and the translation operator $T_\pi$ acts on $A$ by
$$T_{\pi }f(x_1,\ldots,x_n)=f(q^{\frac{1}{2}}x_1,\ldots,q^{\frac{1}{2}}x_n).$$

The Macdonald operator $E_\pi$ can be written as
$$E_\pi f=\sum_{\sigma\in(-1,+1)^n} 
\prod_{i=1}^n\frac {1-tx_{i}^{2\sigma_i}} 
{1-x_{i}^{2\sigma_i}}\,
\prod_{1\le i < j \le n}
\frac {1-t x_{i}^{\sigma_i}x_{j}^{\sigma_j}} 
{1-x_{i}^{\sigma_i}x_{j}^{\sigma_j}}\,
f(q^{\sigma_1/2}x_1,\ldots,q^{\sigma_n/2}x_n).$$
For any dominant weight $\la=\sum_{i=1}^n \la_i \ep_i$, equivalently for any partition $(\la_1,\la_2,\ldots,\la_n)$, the Macdonald polynomial $P_{\la}$ is defined, up to a constant, by
\begin{equation}
E_\pi P_{\la}= e_\la \,P_{\la}\qquad \mathrm{with} \qquad e_\la=\prod_{i=1}^n (q^{\la_i/2}t^{n-i+1}+q^{-\la_i/2}).
\end{equation}

The root system $C_n$ has one minuscule weight $\om_1=\ep_1$ and one quasi-minuscule weight $\om_2=\ep_1+\ep_2$. 

For any partition $\la=(\la_1,\ldots,\la_n)$ and any integer $1 \le i \le n$, we denote by $\la^{(i)}$ (resp. $\la_{(i)}$) the partition $\mu$ (if it exists) such that $\mu_j=\la_j$ for $j\neq i$ and $\mu_i=\la_i +1$ (resp. $\mu_i=\la_i -1$). By Theorem 1 the Pieri formula for $\om_1$ writes as
\[P_{\ep_1}P_{\la}= \sum_{k=1}^{n} (a^+_k P_{\la^{(k)}}+ a^-_k P_{\la_{(k)}}),\]
with
\begin{equation}
\begin{split}
a^+_k=&\prod_{i=1}^{k-1}\frac{1-q^{\la_i-\la_k}t^{k-i-1}}
{1-q^{\la_i-\la_k}t^{k-i}}\,\frac{1-q^{\la_i-\la_k-1}t^{k-i+1}}
{1-q^{\la_i-\la_k-1}t^{k-i}},\\
a^-_k=&\frac{1-q^{\la_k}t^{n-k}}
{1-q^{\la_k}t^{n-k+1}}\,
\frac{1-q^{\la_k-1}t^{n-k+2}}
{1-q^{\la_k-1}t^{n-k+1}}\\\times
&\prod_{\begin{subarray}{c}i=1\\i\neq k\end{subarray}}^n \frac{1-q^{\la_i+\la_k}t^{2n-i-k+1}}
{1-q^{\la_i+\la_k}t^{2n-i-k+2}}\,
\frac{1-q^{\la_i-\la_k-1}t^{2n-i-k+3}}
{1-q^{\la_i-\la_k-1}t^{2n-i-k+2}},\\\times
&\prod_{i=k+1}^{n}\frac{1-q^{\la_k-\la_i}t^{i-k-1}}
{1-q^{\la_k-\la_i}t^{i-k}}\,\frac{1-q^{\la_k-\la_i-1}t^{i-k+1}}
{1-q^{\la_k-\la_i-1}t^{i-k}}.
\end{split}
\end{equation}

\section{A Pieri formula}

In this paper we shall not study the general Pieri formula for $C_n$, which gives the explicit decomposition of the product $P_{r \omega_s}\, P_{\la}$
with $r$ some positive integer, $1\le s\le n$ and $\la$ any dominant weight. 

We shall only consider the particular case where $s=1$ and $\la$ is a multiple of $\om_1$. Nevertheless this will produce a deep result.

For any dominant weight $\la=\sum_{i=1}^n \la_i \ep_i$ we normalize Macdonald polynomials by
\begin{equation}
Q_{\la}=\prod_{1\le i \le j\le n} \frac{(q^{\la_i-\la_j}t^{j-i+1};q)_{\la_j-\la_{j+1}}}{(q^{\la_i-\la_j+1}t^{j-i};q)_{\la_j-\la_{j+1}}}\,P_{\la}.
\end{equation}
Observe that the normalization factor is identical with the non combinatorial expression of $b_\la$, the factor apppearing in the classical normalization $Q_{\la}=b_{\la}P_{\la}$ for Macdonald polynomials of type $A$ (~\cite[pp. 338-339]{Ma2} and e.g.~\cite[Equ. 2.6]{S}). For instance we have
\[Q_{\la_1\ep_1}=\frac{(t;q)_{\la_1}}{(q;q)_{\la_1}}\,P_{\la_1\ep_1},\qquad
Q_{\la_1 \ep_1+\la_2 \ep_2}=\frac{(t;q)_{\la_1-\la_2}}{(q;q)_{\la_1-\la_2}} \, \frac{(t;q)_{\la_2}}{(q;q)_{\la_2}}\,\frac{(q^{\la_1-\la_2}t^2;q)_{\la_2}}{(q^{\la_1-\la_2+1}t;q)_{\la_2}}\,P_{\la_1 \ep_1+\la_2 \ep_2}.\]

We shall prove the following remarkable Pieri formula.
\begin{theo}
For any partition $\la=(\la_1,\la_2)$ we have
\[Q_{\la_1\ep_1}\, Q_{\la_2\ep_1}=\sum_{\begin{subarray}{c}(i,j)\in \mathbb{N}^2\\ 0 \le i+j \le \la_2 \end{subarray}} c_{ij}(\la_1,\la_2) \, Q_{(\la_1+i-j)\ep_1+(\la_2-i-j)\ep_2},\]
with
\begin{equation*}
c_{ij}(\la_1,\la_2)= \frac{(t;q)_i}{(q;q)_i}\,\frac{(t;q)_j}{(q;q)_j}\,
\frac{(q^{\la_1-\la_2+i+1};q)_i}{(q^{\la_1-\la_2+i}t;q)_i}
\frac{(q^{\la_1+\la_2-j-1}t^{2n};1/q)_j}{(q^{\la_1+\la_2-j}t^{2n-1};1/q)_j}.
\end{equation*}
\end{theo}

Here two remarks are needed. Firstly, in~\cite{Mi} and~\cite[Theorem 4]{La} it was shown independently that the generating function of the polynomials $Q_{r\ep_1}$ is given by
\begin{equation}
\sum_{r\in \mathbb{N}} u^r\, Q_{r\ep_1}=
\prod_{i=1}^n \frac{(tux_i;q)_\infty} {(ux_i;q)_\infty}\
\frac{(tu/x_i;q)_\infty} {(u/x_i;q)_\infty}.
\end{equation}
By an obvious two-fold product, Theorem 3 amounts to give the explicit Macdonald development of
\[\prod_{i=1}^n \prod_{j=1}^2 \frac{(tu_jx_i;q)_\infty} {(u_jx_i;q)_\infty}\
\frac{(tu_j/x_i;q)_\infty} {(u_j/x_i;q)_\infty},\]
which is reminiscent of the celebrated Cauchy formula~\cite[(4.13), p. 324]{Ma2} for Macdonald polynomials of type $A$.

Secondly, the Pieri coefficients $c_{ij}(\la_1,\la_2)$ appear to be fully factorized. One may wonder whether this remarkable property keeps verified in the general case 
\[P_{r \ep_1}\, P_{\la}=\sum_{\begin{subarray}{c}\tau\in\Sigma(r\ep_1) \\
\la+\tau\in P^+\end{subarray}} C_\tau \, P_{\la+\tau},\]
with $\la$ arbitrary. However this is not true.

Actually computer calculations, performed for $n=2$, show that for $r=2$ all coefficients $C_\tau$ factorize but $C_0$. For $r=3$ all $C_\tau$ factorize but those with $\tau \in \{\pm \ep_1,\pm \ep_2\}$. For $r=4$ all coefficients factorize but $C_0$ and those for $\tau \in \{\pm 2 \ep_1,\pm 2\ep_2,\pm \ep_1 \pm \ep_2\}$.

In general when $\tau$ belongs to the \textit{boundary} of the convex hull of the Weyl group orbit of $r\ep_1$, $C_\tau$ is fully factorized. For values of $\tau$ \textit{inside} this convex hull, we expect $C_\tau$ to write as a sum of factorized terms, the number of which increases according to the distance of $\tau$ to the boundary.

Of course the above remarks are only valid if the dominant weight $\la$ keeps generic. For specific values of $\la$, the Pieri coefficients may all factorize. Theorem 3 shows that it is the case for $\la=s \om_1$, but we also conjecture the property for $\la=s \om_2$.

\section{Proof of Theorem 3}

The proof relies on the following property giving the action of the Macdonald operator $E_\pi$ on a product $Q_{\la_1\ep_1}Q_{\la_2\ep_1}$. This result will be proved in Section 10. Obviously for $\la_2=0$ we recover (4.2).
\begin{theo}
For any partition $\la=(\la_1,\la_2)$ we have
\begin{multline*}
E_\pi (Q_{(\la_1)}Q_{(\la_2)}) = e_{\la}\,Q_{(\la_1)}Q_{(\la_2)}+
(1-t)\prod_{i=1}^{n-2}(t^i+1) \\\times
\Big(
q^{-\frac{1}{2}(\la_1-\la_2)}t^{n-1} \sum_{k=1}^{\la_2}\, q^{-k}
(1- q^{2k+\la_1-\la_2})\, Q_{(\la_1+k)}Q_{(\la_2-k)}\\
-q^{-\frac{1}{2}(\la_1+\la_2)}\sum_{k=1}^{\la_2}\, t^{k-1} 
(1-q^{-2k+\la_1+\la_2}t^{2n})\,
Q_{(\la_1-k)}Q_{(\la_2-k)}\Big).
\end{multline*}
\end{theo}

\begin{proof}[Proof of Theorem 3] It is done by induction on $\la_2$. Firstly we consider the case $\la_2=1$. Writing (4.3) for  $\la=(r)$, we obtain
\begin{equation*}
Q_{(1)}Q_{(r)}= Q_{(r,1)}+\frac{1-t}{1-q}\frac{1-q^{r+1}}{1-q^{r}t}\,
\, Q_{(r+1)}
+\frac{1-t}{1-q}\frac{1-q^{r-1}t^{2n}}{1-q^{r}t^{2n-1}}\,Q_{(r-1)},
\end{equation*}
which proves Theorem 3 for $\la_2=1$.

In a second step, assuming the property true for $\la_2\le s-1$ we prove it for $\la_2=s$. For this purpose we write the Macdonald development of $Q_{(r)} Q_{(s)}$ as
\[Q_{(r)}Q_{(s)}=\sum_{\mu\in P^+} c_{\mu}(r,s) \, Q_{\mu}.\]
Applying $E_\pi$, Theorem 4 yields
\begin{multline*}
\sum_{\mu\in P^+} c_{\mu}(r,s)\,(e_{\mu}-e_{(r,s)}) \, Q_{\mu}= 
(1-t)\prod_{i=1}^{n-2}(t^i+1) \\\times
\Big(
q^{-\frac{1}{2}(r-s)}t^{n-1} \sum_{k=1}^{s}\, q^{-k}
(1- q^{2k+r-s})\, Q_{(r+k)}Q_{(s-k)}\\
-q^{-\frac{1}{2}(r+s)}\sum_{k=1}^{s}\, t^{k-1} 
(1-q^{-2k+r+s}t^{2n})\,
Q_{(r-k)}Q_{(s-k)}\Big).
\end{multline*}
Then we apply the inductive hypothesis to the products of Macdonald operators on the right-hand side. We obtain at once that the partitions $\mu$ are of the form $(r+i-j,s-i-j)$ with $i+j\le s$.

Moreover by identification of coefficients we have
\begin{multline*}
c_{i,j}(r,s)\,(e_{(r+i-j,s-i-j)}-e_{(r,s)})=
(1-t) \prod_{i=1}^{n-2}(t^i+1)\\\times
\Big(q^{-\frac{1}{2}(r-s)}t^{n-1} \sum_{k=1}^{i}\, q^{-k}
(1- q^{2k+r-s})\, c_{i-k,j}(r+k,s-k)\\
-q^{-\frac{1}{2}(r+s)}\sum_{k=1}^{j}\, t^{k-1} 
(1-q^{-2k+r+s}t^{2n})\,
c_{i,j-k}(r-k,s-k)\Big).
\end{multline*}
Therefore putting $a=q^{r-s}$ and $b=q^{r+s}t^{2n}$, we have only to prove the identity
\begin{multline*}
\frac{(t;q)_i}{(q;q)_i}\,\frac{(t;q)_j}{(q;q)_j}\,
\frac{(aq^{i+1};q)_i}{(aq^{i}t;q)_i}
\frac{(bq^{-j-1};1/q)_j}{(bq^{-j}/t;1/q)_j}\\
\Big(b^{1/2}q^{-j}/t+a^{1/2}q^{i}+a^{-1/2}q^{-i}/t+b^{-1/2}q^{j}-b^{1/2}/t-a^{1/2}-a^{-1/2}/t-b^{-1/2}\Big)=\\
(1-t)/t\Big(a^{-1/2} \sum_{k=1}^{i}\, q^{-k}(1
- aq^{2k})\, \frac{(t;q)_{i-k}}{(q;q)_{i-k}}\,\frac{(t;q)_j}{(q;q)_j}\,\frac{(aq^{i+k+1};q)_{i-k}}{(aq^{i+k}t;q)_{i-k}}
\frac{(bq^{-j-1};1/q)_j}{(bq^{-j}/t;1/q)_j}\\
-b^{-1/2}\sum_{k=1}^{j}\, t^k 
(1-bq^{-2k})\,
\frac{(t;q)_i}{(q;q)_i}\,\frac{(t;q)_{j-k}}{(q;q)_{j-k}}\,
\frac{(aq^{i+1};q)_i}{(aq^{i}t;q)_i}
\frac{(bq^{-j-k-1};1/q)_{j-k}}{(bq^{-j-k}/t;1/q)_{j-k}}
\Big).
\end{multline*}
The latter is a consequence of the stronger result
\begin{multline*}
\frac{(t;q)_i}{(q;q)_i}\,
\frac{(aq^{i+1};q)_i}{(aq^{i}t;q)_i}\,
\Big(a^{1/2}q^{i}+a^{-1/2}q^{-i}/t-a^{1/2}-a^{-1/2}/t\Big)=\\
(1-t)a^{-1/2}/t \sum_{k=1}^{i}\, q^{-k}(1
- aq^{2k})\, \frac{(t;q)_{i-k}}{(q;q)_{i-k}}\,\frac{(aq^{i+k+1};q)_{i-k}}{(aq^{i+k}t;q)_{i-k}},
\end{multline*}
and its analog obtained by substituting $(b,j,1/q,1/t)$ to $(a,i,q,t)$. 

This identity is easily proved, because it can be transformed into
\[\sum_{k=1}^i q^{i-k}(1-aq^{2k})\frac{(q^{i-k+1};q)_k}{(q^{i-k}t;q)_k} \,\frac{(aq^it;q)_k}{(aq^{i+1};q)_k}
=\frac{1-q^i}{1-t}\,(1-aq^it),\]
which is the classical summation formula
\[{}_6\phi_5\left[\begin{matrix}  a,qa^{\frac{1}{2}},-qa^{\frac{1}{2}},q,aq^it,q^{-i} \\
a^{\frac{1}{2}},-a^{\frac{1}{2}},a,q^{1-i}/t,aq^{i+1} \end{matrix};q,1/t \right]=\frac{1-aq^i}{1-a}\frac{1-q^{-i}/t}{1-1/t},\]
for a terminating very-well-poised ${}_6\phi_5$ basic hypergeometric series~\cite[(2.4.2)]{GR}.
\end{proof}

\section{An inverse Pieri formula}

In view of (5.2), the following result completely determines the ``two-row'' Macdonald polynomials of type $C_n$ (hence all Macdonald polynomials for $C_2$).
\begin{theo}
For any partition $\la=(\la_1,\la_2)$ we have
\[Q_{\la_1\ep_1+\la_2\ep_2}=\sum_{\begin{subarray}{c}(i,j)\in \mathbb{N}^2\\ 0 \le i+j \le \la_2 \end{subarray}} C_{ij}(\la_1,\la_2) \, Q_{(\la_1+i-j)\ep_1} \,Q_{(\la_2-i-j)\ep_1},\]
with
\begin{multline*}
C_{ij}(\la_1,\la_2)=t^{i+j}\, \frac{(1/t;q)_i}{(q;q)_i}\,\frac{(1/t;q)_j}{(q;q)_j}\,
\frac{(q^{\la_1-\la_2+1};q)_i}{(q^{\la_1-\la_2+1}t;q)_i}\,\frac{1-q^{\la_1-\la_2+2i}}{1-q^{\la_1-\la_2+i}}\\\times
\frac{(q^{\la_1+\la_2-1}t^{2n};1/q)_j}{(q^{\la_1+\la_2-1}t^{2n-1};1/q)_j}
\,\frac{1-q^{\la_1+\la_2-2j}t^{2n}}{1-q^{\la_1+\la_2-j}t^{2n}}.
\end{multline*}
\end{theo}

Starting from Theorem 3, the proof is obtained by inverting infinite multi-dimensional matrices.

An infinite one-dimensional matrix $(f_{ij})_{i,j\in \mathbb{Z}}$ is said to 
be lower-triangular if $f_{ij}=0$ unless $i\ge j$. Two infinite lower-triangular 
matrixes $(f_{ij})_{i,j\in \mathbb{Z}}$ and $(g_{kl})_{k,l\in \mathbb{Z}}$ are said to be mutually inverse if 
$\sum_{i\ge j \ge k} f_{ij}g_{jk}=\delta_{ik}$.

The simplest case of such a pair is given by Bressoud's matrix inverse~\cite{B}, which states that, defining
$$A_{ij}(u,v)=(u/v)^j\ \frac{(u/v;q)_{i-j}}{(q;q)_{i-j}} \
\frac{(u;q)_{i+j}}{(vq;q)_{i+j}}\ 
\frac{1-vq^{2j}}{1-v},$$
the matrices $A(u,v)$ and $A(v,u)$ are mutually inverse.
We refer to~\cite{K} for a generalization 
of~\cite{B} and to~\cite{La} for some applications.

An equivalent formulation of Bressoud's matrix inverse is obtained by considering the matrices
\begin{equation*}
\begin{split}
f_{ij}&=(u/v)^{i-j}\,\frac{(v/u;q)_{i-j}}{(q;q)_{i-j}}\,\frac{(vq^{2j};q)_{i-j}}{(uq^{2j+1};q)_{i-j}}\,\frac{1-vq^{2i}}{1-vq^{2j}},\\
g_{kl}&=\frac{(u/v;q)_{k-l}}{(q;q)_{k-l}}\,\frac{(vq^{k+l+1};q)_{k-l}}{(uq^{k+l};q)_{k-l}},
\end{split}
\end{equation*}
which are mutually inverse because we have
\begin{equation*}
\begin{split}
f_{ij}&=(u/v)^{i}\,\frac{(u;q)_{2j}}{(vq;q)_{2j}}\,\frac{1-vq^{2i}}{1-v}\, A_{ij}(v,u),\\
g_{kl}&=(v/u)^{l}\,\frac{(vq;q)_{2k}}{(u;q)_{2k}}\,\frac{1-v}{1-vq^{2l}}\, A_{kl}(u,v).
\end{split}
\end{equation*}

Mutually inverse matrices can be also considered in the more general context of multi-dimensional infinite matrices. In this paper we shall only deal with dimension two. We refer to~\cite{LS,LS2} for more general results and references. 

Given multi-integers $\mathbf{m}= (m_1,m_2) \in \mathbb{Z}^2$, we write 
$\mathbf{m}\ge\mathbf{k}$ for $m_1\ge k_1$, $m_2\ge k_2$. An infinite two-dimensional matrix 
$F=(f_{\mathbf{j}\mathbf{k}})_{\mathbf{j},\mathbf{k}\in{\mathbb Z}^2}$ 
is said to be lower-triangular if $f_{\mathbf{j}\mathbf{k}}=0$ unless 
$\mathbf{j}\ge\mathbf{k}$. When all $f_{\mathbf{k k}}\ne 0$, 
there exists a unique lower-triangular matrix
$G=(g_{\mathbf{k}\mathbf{l}})_{\mathbf{k},\mathbf{l}\in{\mathbb Z}^2}$, called the inverse of $F$, such that 
\begin{equation*}
\sum_{\mathbf{j}\ge\mathbf{k}\ge\mathbf{l}}
f_{\mathbf{j}\mathbf{k}}\,
g_{\mathbf{k}\mathbf{l}}=\delta_{\mathbf{jl}},
\end{equation*}
for all $\mathbf{j},\mathbf{l}\in{\mathbb Z}^2$,
where $\delta_{\mathbf{jl}}$ is the usual Kronecker symbol.

Starting from the one-dimensional Bressoud's pair with $u/v=t$, and taking a two-fold product, we easily obtain the following pair of mutually inverse two-dimensional matrices
\begin{equation*}
\begin{split}
f_{\mathbf{j}\mathbf{k}}&=t^{j_1-k_1}
\,\frac{(1/t;q)_{j_1-k_1}}{(q;q)_{j_1-k_1}}
\,\frac{(u_1q^{2k_1};q)_{j_1-k_1}}{(tu_1q^{2k_1+1};q)_{j_1-k_1}}
\,\frac{1-u_1q^{2j_1}}{1-u_1q^{2k_1}}\\
&\hspace{2cm}\times t^{j_2-k_2}\,
\frac{(1/t;q)_{j_2-k_2}}{(q;q)_{j_2-k_2}}
\frac{(u_2q^{2k_2};q)_{j_2-k_2}}{(tu_2q^{2k_2+1};q)_{j_2-k_2}}
\,\frac{1-u_2q^{2j_2}}{1-u_2q^{2k_2}},\\
g_{\mathbf{k}\mathbf{l}}&=\frac{(t;q)_{k_1-l_1}}{(q;q)_{k_1-l_1}}
\,\frac{(u_1q^{k_1+l_1+1};q)_{k_1-l_1}}{(tu_1q^{k_1+l_1};q)_{k_1-l_1}}
\,\frac{(t;q)_{k_2-l_2}}{(q;q)_{k_2-l_2}}
\,\frac{(u_2q^{k_2+l_2+1};q)_{k_2-l_2}}{(tu_2q^{k_2+l_2};q)_{k_2-l_2}}.
\end{split}
\end{equation*}

\begin{proof}[Proof of Theorem 5]
For $\mathbf{j}=(j_1,j_2)$ and $\mathbf{k}=(k_1,k_2)$, we define
\begin{equation*}
\begin{split}
a_{\mathbf{j}}&=Q_{(\la_1+j_1-j_2)\ep_1}\,Q_{(\lambda_2-j_1-j_2)\ep_1},\\
b_{\mathbf{k}}&=Q_{(\la_1+k_1-k_2)\ep_1+(\la_2-k_1-k_2)\ep_2}.
\end{split}
\end{equation*}
If in Theorem 3 we perform the substitutions $\la_1\mapsto \la_1+l_1-l_2$, $\la_2\mapsto \la_2-l_1-l_2$, 
we easily obtain 
\[a_{\mathbf{l}}=\sum_{\mathbf{k}\ge\mathbf{l}}(t/q)^{k_2-l_2}\,g_{\mathbf{k}\mathbf{l}}\,b_{\mathbf{k}},\]
with $u_1= q^{\la_1-\la_2}$ and $u_2= q^{-\la_1-\la_2}t^{-2n}$. This yields
\[b_{\mathbf{k}}=\sum_{\mathbf{j}\ge\mathbf{k}}(t/q)^{j_2-k_2}\,f_{\mathbf{j}\mathbf{k}}\,a_{\mathbf{j}}.\]
We conclude by the substitution $k_1=k_2=0$.
\end{proof}

Our expansion of $Q_{\la_1\ep_1+\la_2\ep_2}$ is very similar to the one obtained by Jing and J\'{o}zefiak~\cite{JJ} for the two-row Macdonald polynomials associated with $A_n$. The latter is also a consequence of Bressoud's matrix inverse. 

For $t=q$ we recover the following classical result due to Hermann Weyl~\cite{W}.
\begin{coro} 
The irreducible characters $\chi_{\la_1\ep_1+\la_2\ep_2}$ of the symplectic group $\mathrm{Sp}(2n,\mathbb{C})$ are given by
\[\chi_{\la_1\ep_1+\la_2\ep_2}=
 h_{\la_1}\, h_{\la_2}+ h_{\la_1}\, h_{\la_2-2}
-h_{\la_1+1}\,h_{\la_2-1}-h_{\la_1-1}\,h_{\la_2-1},\]
with $h_k$ the complete functions defined by 
\[\sum_{k\ge 0} u^k h_k =\prod_{i=1}^n \frac{1}{(1-ux_i)(1-u/x_i)}.\]
\end{coro}
\begin{proof}
If $t=q$ we have $Q_{r\ep_1}=h_r$ and $P_{\la_1\ep_1+\la_2\ep_2}=Q_{\la_1\ep_1+\la_2\ep_2}$. In the summation of Theorem 5, the only non-zero contributions correspond to $i,j=0,1$.
\end{proof}

\section{A rational identity}

It remains to prove Theorem 4, which will be done in Section 10. For that purpose, two ingredients will be needed. Firstly the language of $\la$-rings, a powerful way to handle series, which we briefly recall in Section 9. Secondly a remarkable rational identity, which we present in this section.

This multivariate identity depends on two sets of indeterminates $x=(x_1,\ldots,x_n)$ and $u=(u_1,\ldots,u_r)$. In this paper we shall only use it for $r=2$. However in view of its own interest, we prove it in full generality. 

By specialization of the indeterminates $x$ or $u$, we may obtain basic $q$-hypergeometric identities. We present some examples of the latter at the end of this paper, in Section 12.

We start from the following identity, which was proved independently in~\cite[Lemma 2]{Mi} and~\cite[Theorem 5]{La}. Let $x=(x_1,\ldots,x_n)$ and $u$ be $n+1$ indeterminates. We have 
\begin{multline*}
\sum_{\sigma\in (-1,+1)^n} 
\prod_{i=1}^n \frac{1-tx_{i}^{2\ta_i}} 
{1-x_{i}^{2\ta_i}}\,
\frac{1-tux_i^{-\ta_i}}
{1-ux_i^{-\ta_i}}\,
\prod_{1\le i < j \le n}
\frac {1-tx_{i}^{\ta_i}x_{j}^{\ta_j}} 
{1-x_{i}^{\ta_i}x_{j}^{\ta_j}}=\\
\prod_{i=1}^{n-1}(t^i+1) \left(t^n+
\prod_{i=1}^n \frac{1-tux_i}
{1-ux_i}\, \frac{1-tu/x_i}{1-u/x_i}\right).
\end{multline*}
Recalling the definition (4.1) of $\Phi_\pi$, writing $T_{i}$ for the operator $x_i\rightarrow1/x_i$ and $\mathcal{T}= (1+T_1)\cdots(1+T_n)$, this identity can be written more compactly
\begin{equation}
\mathcal{T}\left(\Phi_\pi \, \prod_{i=1}^n
\frac{1-tu/x_i}{1-u/x_i}  \right)=
\prod_{i=1}^{n-1}(t^i+1) \, (t^n+R(u)),
\end{equation}
with
\[R(u)=\prod_{i=1}^n \frac{1-tux_i}{1-ux_i}\,
\frac{1-tu/x_i}{1-u/x_i}.\]
In particular for $u=0$ we have
\begin{equation}
\mathcal{T}\left(\Phi_\pi\right)=
\prod_{i=1}^{n}(t^i+1).
\end{equation}

Both properties correspond to the cases $r=0,1$ of the following result.
\begin{theo}
Let $x=(x_1,\ldots,x_n)$ and $u=(u_1,\ldots,u_r)$. We have
\begin{equation}
\mathcal{T}\left(\Phi_\pi \,  \prod_{k=1}^r \, \prod_{i=1}^n
\frac{1-tu_k/x_i}{1-u_k/x_i}  \right)=
(-t;t)_{n-r} \, \sum_{I \subset \{1,\ldots,r \}} c_I \prod_{i \in I} R(u_i),
\end{equation}
with 
$$c_I=t^{n(r-|I|)}\prod_{1\le i < j \le r}\frac{1-v_iv_j}{1-tv_iv_j},$$
defining $v_i=u_i$ if $i\in I$ and $v_i=1/tu_i$ if $i\notin I$.
\end{theo}

\noindent \textit{Remark:} The value of the factor $(-t;t)_{n-r}$ is $\prod_{k=1}^{n-r}(t^k+1)$ or $\prod_{k=0}^{r-n-1}(t^{-k}+1)^{-1}$ according to the sign of $n-r$.

\begin{proof}\textit{(i) First case: $r\le n$.} The proof is done by induction on $r$, starting from the case $r=1$ given by (7.1). Both sides of the identity are rational functions of $u_r$ having poles firstly at $u_r=x_i$ and $u_r=1/x_i$ for $i=1,\ldots,n$, secondly at $u_r=u_i$ and $u_r=1/tu_i$ for $i=1,\ldots,r-1$. 

Their constant terms are equal as a consequence of the inductive hypothesis. Actually, specifying the dependence on $r$, for any set $I\subset \{1,\ldots,r-1\}$ we have
\[(c^{(r)}_I+c^{(r)}_{I\cup r})\arrowvert_{u_r=0} =(t^{n-r+1}+1)c^{(r-1)}_I.\]
Thus at $u_r=0$ we obtain the identity written for $r-1$. Therefore it is sufficient to prove that both sides of the identity have the same residue at each of their poles. 

In a first step, we consider the poles $u_r=x_i$ and $u_r=1/x_i$. By symmetry, the equality of residues has only to be checked for some $x_i$, say at $u_r=x_n$ or $u_r=1/x_n$. We shall only do it at $u=x_n$, the proof at $u=1/x_n$ being similar.

For any indeterminates $(a_1,\ldots,a_m)$ we have 
$$\prod_{i=1}^{m} \frac{tu-a_i}{u-a_i}=t^m+
(t-1)\sum_{i=1}^{m}\frac{a_i}{u-a_i} \
\prod_{\begin{subarray}{c}j=1\\j\neq i \end{subarray}}^m
\frac{ta_i-a_j}{a_i-a_j}.$$
This decomposition as a sum of partial fractions is actually 
a Lagrange interpolation~\cite[Section 7.8]{Las}. 
We first apply it to $R(u)$. Its residue 
at $u=x_n$ is given by
$$x_n\, (t-1)\, \prod_{j=1}^{n-1}
\frac{tx_n-x_j}{x_n-x_j} 
\prod_{j=1}^{n}\frac{1-tx_nx_j}{1-x_nx_j}.$$
The residue of the right-hand side at $u_r=x_n$ is therefore
\[x_n (t-1)\,\prod_{k=1}^{n-r}(t^k+1)  \left(\sum_{I \subset \{1,\ldots,r-1 \}} c_{I\cup r}\arrowvert_{u_r=x_n} \prod_{i \in I} R(u_i)\right)
\prod_{j=1}^{n-1}
\frac{tx_n-x_j}{x_n-x_j} 
\prod_{j=1}^{n}\frac{1-tx_nx_j}{1-x_nx_j}.\]
Then we apply the Lagrange interpolation to
$$\prod_{i=1}^n \frac{tu_r-x_i^{\ta_i}}{u_r-x_i^{\ta_i}},$$
on the left-hand side of the identity.
Only fractions with $\ta_n=1$ contribute to the residue 
at $u_r=x_n$. Thus it can be written as
\begin{multline*}
x_{n}(t-1)\frac {1-tx_{n}^{2}} 
{1-x_{n}^{2}}\, \prod_{k=1}^{r-1}
\frac{1-tu_k/x_n}{1-u_k/x_n}
\sum_{\ta\in(-1,+1)^{n-1}} 
\left(\prod_{k=1}^{r-1} \, \prod_{l=1}^{n-1}
\frac{1-tu_kx_l^{-\ta_l}}{1-u_kx_l^{-\ta_l}}\right)\\ \times
\prod_{i=1}^{n-1}\frac {1-tx_{i}^{2\ta_i}} 
{1-x_{i}^{2\ta_i}}\, 
\frac{1-tx_{n}x_i^{-\ta_i}}
{1-x_{n}x_i^{-\ta_i}}\,
\frac {1-tx_{i}^{\ta_i}x_{n}} 
{1-x_{i}^{\ta_i}x_{n}}\,
\prod_{1\le i < j \le n-1}
\frac {1-tx_{i}^{\ta_i}x_{j}^{\ta_j}} 
{1-x_{i}^{\ta_i}x_{j}^{\ta_j}}.
\end{multline*}
Now the product
$$\prod_{i=1}^{n-1} \frac{1-tx_nx_i}{1-x_nx_i}\, 
\frac{1-tx_n/x_i}{1-x_n/x_i}$$
is obviously invariant under any $T_i$, and can be cancelled on both sides. Therefore, writing $x_n=z$, we are led to prove the following equality, for $r-1$ indeterminates $(u_1,\ldots,u_{r-1})$ and $n-1$ variables $(x_1,\ldots,x_{n-1})$,
\begin{multline*}
\prod_{k=1}^{n-r}(t^k+1)  \left(\sum_{I \subset \{1,\ldots,r-1\}} c_{I\cup r}\arrowvert_{u_r=z} \prod_{i \in I} 
\Big( \frac{1-tu_iz}{1-u_iz} \,
\frac{1-tu_i/z}{1-u_i/z}\, R(u_i)\Big) \right)\\ =
\prod_{k=1}^{r-1}
\frac{1-tu_k/z}{1-u_k/z}\,
\mathcal{T}\left(\Phi_\pi \,\prod_{k=1}^{r-1} \, \prod_{i=1}^{n-1}
\frac{1-tu_k/x_i}{1-u_k/x_i} \right).
\end{multline*}
Specifying the dependence on $n,r$ it is equivalent to check that for any $I \subset \{1,\ldots,r-1 \}$ we have
\begin{equation}
c_{I\cup r}^{(n,r)}=c_{I}^{(n-1,r-1)}
\prod_{i\in I} \frac{1-u_iu_r}{1-tu_iu_r}\,
\prod_{j\notin I} \frac{1-tu_j/u_r}{1-u_j/u_r},
\end{equation}
which is obvious.

In a second step, we consider the values $u_r=u_i$ and $u_r=1/tu_i$ for $i=1,\ldots,r-1$, which are poles for the right-hand side of the identity, but not for its left-hand side. Therefore we have to prove that their residue is zero. The proofs being similar, we shall only
do it for $u_r=1/tu_i$.

The only $c_I$ having a pole at $u_r=1/tu_i$ correspond to sets $I$ such that $r,i\in I$ or $r,i\notin I$. The sum of both contributions brings a factor
\[t^{n(r-|I|)-1}\frac{1-t^2u_iu_r}{1-tu_iu_r}+
t^{n(r-|I|-2)}\frac{1-u_iu_r}{1-tu_iu_r}R(u_i)R(u_r),\]
for any set $I \subset \{1,\ldots,r-2\}$. Since $R(u_i)R(1/tu_i)=t^{2n}$, its residue is
\[\frac{t-1}{tu_i} \,(t^{n(r-|I|-1)}-t^{n(r-|I|-2)}t^{2n}/t)=0.\]
Summing all contributions, the residue of the right-hand side at $u_r=1/tu_i$ is zero. This achieves the proof for $r\le n$.

\noindent \textit{(ii) Second case: $n<r$.} The proof is strictly identical to the previous one except that, once (7.4) obtained, it is done by induction on $n$ rather than on $r$. Therefore it remains to prove the identity for $n=0$, which writes as
\begin{equation}
\sum_{I \subset \{1,\ldots,r \}} c_I =\prod_{k=0}^{r-1}(t^{-k}+1).
\end{equation}
The left-hand side is a rational function of $u_r$ having poles at $u_r=u_i$ and $u_r=1/tu_i$ for $i=1,\ldots,r-1$. Exactly as above, its residues at these poles are shown to be zero. Thus it only remains to show (7.5) at $u_r=0$. 

This is done by induction on $r$ starting from the obvious case $r=1$. Specifying the dependence on $n$, for any $I \subset \{1,\ldots,r-1 \}$ we have
\[(c^{(r)}_I+c^{(r)}_{I\cup r})\arrowvert_{u_r=0} =(t^{1-r}+1)c^{(r-1)}_I.\]
Therefore at $u_r=0$ we obtain (7.5) written for $r-1$.
\end{proof}

In this paper we shall only use Theorem 6 for $r=2$.
\begin{coro}
For any two indeterminates $u,v$, we have
\begin{multline}
\mathcal{T}\left(\Phi_\pi \,  \prod_{i=1}^n
\frac{1-tu/x_i}{1-u/x_i}\,\frac{1-tv/x_i}{1-v/x_i}  \right)=
\prod_{i=1}^{n-2}(t^i+1) \\ \times \Big( t^{2n-1} \frac{1-t^2uv}{1-tuv}
+ t^{n-1} \frac{u-tv}{u-v} R(u)
+ t^{n-1} \frac{v-tu}{v-u} R(v)+ \frac{1-uv}{1-tuv} R(u)R(v) \Big).
\end{multline}
\end{coro}

\section{$\la$-rings}

In the next section, Theorem 4 will be proved by using the powerful language of 
$\la$-rings. Here we only give a short survey of this theory. More details and other applications may be found, for instance,
in~\cite{Las,La1} and in some examples of~\cite{Ma2} (see pp. 25, 43, 65 and 79).

Let $A=\{a_1,a_2,a_3,\ldots\}$ be a (finite or infinite) set of
independent indeterminates, called an alphabet. We introduce the generating functions
$$
E_u(A) = \prod_{a\in A} (1+ua), \quad
H_u(A) =  \prod_{a\in A}  \frac{1}{1-ua}, \quad
P_u(A) = \sum_{a\in A} \frac{a}{1-ua},$$
whose development defines symmetric functions known as elementary functions $e_k(A)$, complete functions $h_{k}(A)$, and 
power sums $p_k(A)$, respectively. Each of these three sets generate algebraically the symmetric algebra $\mathbb{S}(A)$.

We define an action $f \rightarrow f[\,\cdot\,]$ of $\mathbb{S}(A)$ on the ring 
$\mathbb{R}[A]$ of polynomials in $A$ with real coefficients. Since the power sums $p_{k}$ generate $\mathbb{S}(A)$, it is enough to define the action of $p_{k}$ on $\mathbb{R}[A]$. Writing any polynomial as 
$\sum_{c,P} c P$, with $c$ a real constant and $P$ a monomial 
in $(a_{1},a_{2},a_{3},\ldots)$, we define
$$p_{k} \left[\sum_{c,P} c P\right]=\sum_{c,P} c P^{k}.$$
This action extends to $\mathbb{S}[A]$. For instance we obtain
$$ E_u \left[\sum_{c,P} c P\right]= \prod_{c,P} (1+u P)^c, \qquad 
H_u \left[\sum_{c,P} c P\right]= \prod_{c,P} (1-u P)^{-c}.$$
More generally, we can define an action of $\mathbb{S}(A)$ on the ring of rational functions, and even on the ring of formal series, by writing
$$ p_{k} \left( \frac{\sum c P}{\sum d Q}\right) = 
\frac{\sum c P^k}{\sum d Q^k},  $$
with $c,d$ real constants and $P,Q$ monomials in 
$(a_{1},a_{2},a_{3},\ldots)$. This action still extends to $\mathbb{S}(A)$.

If we write $A^{\dag}=\sum_{i} a_i$, we have
$p_{k} [A^{\dag}]=\sum_{i} a_i^k$ by definition. Thus $p_{k} 
[A^{\dag}]=p_{k}(A)$, which yields that for any symmetric function $f$, we have $f[A^\dag]=f(A)$.
In particular
$$f(1,q,q^2,\ldots,q^{m-1})= f 
\left[\frac{1-q^m}{1-q}\right],
\qquad f(1,q,q^2,q^3,\ldots)=f\left[\frac{1}{1-q}\right].$$
Moreover for any formal series $P,Q$ we have
$$h_r[P+Q]= \sum_{k=0}^r h_{r-k} [P] \ h_k [Q], \qquad
e_r[P+Q]= \sum_{k=0}^r e_{r-k} [P] \ e_k [Q].$$
Or equivalently
\begin{equation*}
\begin{split}
&H_u[P+Q]=H_u[P]\,H_u[Q], \qquad \hspace{.45cm} E_u[P+Q]=E_u[P]\,E_u[Q]\\
&H_u[P-Q]=H_u[P]\,{H_u[Q]}^{-1},\qquad E_u[P-Q]=E_u[P]\,{E_u[Q]}^{-1}.
\end{split}
\end{equation*}

As an application, for a finite alphabet 
$X=\{x_1,x_2,\ldots,x_m\}$ and two indeterminates $a,q$ we may write
\begin{equation*}
\begin{split}
H_u \left[\frac{aX^\dag}{1-q}\right]&=\prod_{i\geq0} 
H_u \big[aq^iX^\dag\big] 
=\prod_{k=1}^{m} \ \prod_{i\geq0} H_u \big[aq^ix_k\big]\\
&=\prod_{k=1}^{m} \ \prod_{i\geq0} \frac{1}{1-auq^ix_k}
=\prod_{k=1}^{m} \frac{1}{{(aux_k;q)}_{\infty}}.
\end{split}
\end{equation*}
Since
$$H_u \left[\frac{1-t}{1-q}X^\dag\right]= 
H_u \left[\frac{X^\dag}{1-q}\right] 
{\left(H_u \left[\frac{tX^\dag}{1-q}\right]\right)}^{-1},$$
we obtain
$$H_u\left[\frac{1-t}{1-q}\,X^\dag\right]= \sum_{r\ge 0} 
u^r \,h_{r}\left[\frac{1-t}{1-q}\,X^\dag\right]
= \prod_{i=1}^{m} \frac{{(tux_i;q)}_{\infty}}{{(ux_i;q)}_{\infty}}.$$

\section{Action of the Macdonald operator}

We are now in a position to prove Theorem 4. Recall that in~\cite{Mi} and~\cite[Theorem 4]{La} it was shown independently that the generating function of the polynomials $Q_{(r)}$ is given by
\begin{equation*}
\sum_{r\in \mathbb{N}} u^r\, Q_{(r)}=
\prod_{i=1}^n \frac{(tux_i;q)_\infty} {(ux_i;q)_\infty}\
\frac{(tu/x_i;q)_\infty} {(u/x_i;q)_\infty}
=H_u\left[\frac{1-t}{1-q}X^{\dag}\right],
\end{equation*}
with $X=\{x_1,\ldots,x_n\}\cup \{1/x_1,\ldots,1/x_n\},$
so that $X^{\dag}=\sum_{i=1}^n (x_i+1/x_i)$.

\begin{proof}[Proof of Theorem 4] 
We have to compute the action of the Macdonald operator $E_\pi$ on products $Q_{(\la_1)}Q_{(\la_2)}$, hence the generating function
\[\sum_{\la_1,\la_2\in \mathbb{N}} u^{\la_1}v^{\la_2}\, E_\pi (Q_{(\la_1)}Q_{(\la_2)})=
E_\pi  \left(H_u\left[\frac{1-t}{1-q}X^{\dag}\right]H_v\left[\frac{1-t}{1-q}X^{\dag}\right]\right).\]
By the definition of $E_\pi$ and $X^{\dag}$, this is
\begin{multline*}
\sum_{\sigma\in(-1,+1)^n} 
\prod_{i=1}^n\frac {1-tx_{i}^{2\sigma_i}} 
{1-x_{i}^{2\sigma_i}}\,
\prod_{1\le i < j \le n}
\frac {1-t x_{i}^{\sigma_i}x_{j}^{\sigma_j}} 
{1-x_{i}^{\sigma_i}x_{j}^{\sigma_j}}\\
\times H_u\left[\frac{1-t}{1-q}
\sum_{i=1}^n 
\left(x_iq^{\ta_i/2}+\frac{1}{x_iq^{\ta_i/2}}\right)\right]
H_v\left[\frac{1-t}{1-q}
\sum_{i=1}^n 
\left(x_iq^{\ta_i/2}+\frac{1}{x_iq^{\ta_i/2}}\right)\right].
\end{multline*}
Now we have the relations
\begin{equation*}
xq^{\ta/2}+\frac{1}{xq^{\ta/2}}=
q^{\frac{1}{2}}\left(x+\frac{1}{x}\right)+q^{-\frac{1}{2}}(1-q)x^{-\ta},
\end{equation*}
which are checked separately for $\ta=\pm1$. They imply
\begin{equation*}
\begin{split}
H_u\left[\frac{1-t}{1-q}
\sum_{i=1}^n 
\left(x_iq^{\ta_i/2}+\frac{1}{x_iq^{\ta_i/2}}\right)\right]&=
H_u\left[\frac{1-t}{1-q}q^{\frac{1}{2}}X^{\dag}+
(1-t)q^{-\frac{1}{2}}\sum_{i=1}^n x_i^{-\ta_i}\right]\\
&=H_u\left[(1-t)q^{-\frac{1}{2}}\sum_{i=1}^n x_i^{-\ta_i}\right] 
\, H_u\left[\frac{1-t}{1-q}q^{\frac{1}{2}}X^{\dag}\right]\\
&=\prod_{i=1}^n \frac{1-q^{-\frac{1}{2}}tux_i^{-\ta_i}}
{1-q^{-\frac{1}{2}}ux_i^{-\ta_i}}\,
H_u\left[\frac{1-t}{1-q}q^{\frac{1}{2}}X^{\dag}\right].
\end{split}
\end{equation*}
Finally we have
\begin{multline*}
\sum_{\la_1,\la_2\in \mathbb{N}} u^{\la_1}v^{\la_2}\, E_\pi (Q_{(\la_1)}Q_{(\la_2)})=\sum_{\sigma\in(-1,+1)^n} 
\prod_{i=1}^n\frac {1-tx_{i}^{2\sigma_i}} 
{1-x_{i}^{2\sigma_i}}\,
\prod_{1\le i < j \le n}
\frac {1-t x_{i}^{\sigma_i}x_{j}^{\sigma_j}} 
{1-x_{i}^{\sigma_i}x_{j}^{\sigma_j}}\\
\prod_{i=1}^n \frac{1-q^{-\frac{1}{2}}tux_i^{-\ta_i}}
{1-q^{-\frac{1}{2}}ux_i^{-\ta_i}}\,
\prod_{i=1}^n \frac{1-q^{-\frac{1}{2}}tvx_i^{-\ta_i}}
{1-q^{-\frac{1}{2}}vx_i^{-\ta_i}}\,
H_u\left[\frac{1-t}{1-q}q^{\frac{1}{2}}X^{\dag}\right]
H_v\left[\frac{1-t}{1-q}q^{\frac{1}{2}}X^{\dag}\right].
\end{multline*}
The right hand-side is 
\[\mathcal{T}\left(\Phi_\pi \,  \prod_{i=1}^n
\frac{1-q^{-\frac{1}{2}}tu/x_i}{1-q^{-\frac{1}{2}}u/x_i}\,
\frac{1-q^{-\frac{1}{2}}tv/x_i}{1-q^{-\frac{1}{2}}v/x_i}  \right)
H_u\left[\frac{1-t}{1-q}q^{\frac{1}{2}}X^{\dag}\right]
H_v\left[\frac{1-t}{1-q}q^{\frac{1}{2}}X^{\dag}\right].\]
Therefore we may apply the $r=2$ case (8.6) of Theorem 6, with $q^{-\frac{1}{2}}u$ and $q^{-\frac{1}{2}}v$ instead of $u$ and $v$. We obtain
\begin{multline*}
\sum_{\la_1,\la_2\in \mathbb{N}} u^{\la_1}v^{\la_2}\, E_\pi (Q_{(\la_1)}Q_{(\la_2)})=
\prod_{i=1}^{n-2}(t^i+1) \,H_u\left[\frac{1-t}{1-q}q^{\frac{1}{2}}X^{\dag}\right]
H_v\left[\frac{1-t}{1-q}q^{\frac{1}{2}}X^{\dag}\right] \\
\Big( t^{2n-1} \frac{1-t^2uv/q}{1-tuv/q}
+ t^{n-1} \frac{u-tv}{u-v} R(q^{-\frac{1}{2}}u)\\
+ t^{n-1} \frac{v-tu}{v-u} R(q^{-\frac{1}{2}}v)
+ \frac{1-uv/q}{1-tuv/q} R(q^{-\frac{1}{2}}u)R(q^{-\frac{1}{2}}v) \Big).
\end{multline*}
But we have
\begin{equation*}
\begin{split}
R(q^{-\frac{1}{2}}u)H_u\left[\frac{1-t}{1-q}q^{\frac{1}{2}}X^{\dag}\right]&=\prod_{i=1}^n \frac{1-q^{-\frac{1}{2}}tux_i}{1-q^{-\frac{1}{2}}ux_i}\,
\frac{1-q^{-\frac{1}{2}}tu/x_i}{1-q^{-\frac{1}{2}}u/x_i}\,
H_u\left[\frac{1-t}{1-q}q^{\frac{1}{2}}X^{\dag}\right]\\
&=H_u\big[(1-t)q^{-\frac{1}{2}}X^{\dag}\big]\,
H_u\left[\frac{1-t}{1-q}q^{\frac{1}{2}}X^{\dag}\right]\\
&=H_u\left[\frac{1-t}{1-q}q^{\frac{1}{2}}X^{\dag}+(1-t)q^{-\frac{1}{2}}X^{\dag}\right]\\
&=H_u\left[\frac{1-t}{1-q}q^{-\frac{1}{2}}X^{\dag}\right].
\end{split}
\end{equation*}
Finally writing 
$A=(1-t)/(1-q)\,X^{\dag}$ for a better display, we have
\begin{equation*}
\begin{split}
\sum_{\la_1,\la_2\in \mathbb{N}} u^{\la_1}v^{\la_2}\, &E_\pi (Q_{(\la_1)}Q_{(\la_2)})=
\prod_{i=1}^{n-2}(t^i+1)\\ \times
&\Big( t^{2n-1} \frac{1-t^2uv/q}{1-tuv/q}H_u[q^{\frac{1}{2}}A]
H_v[q^{\frac{1}{2}}A] 
+ t^{n-1} \frac{u-tv}{u-v} H_u[q^{-\frac{1}{2}}A]
H_v[q^{\frac{1}{2}}A]\\
&+ t^{n-1} \frac{v-tu}{v-u} H_u[q^{\frac{1}{2}}A]
H_v[q^{-\frac{1}{2}}A]
+ \frac{1-uv/q}{1-tuv/q} 
H_u[q^{-\frac{1}{2}}A]
H_v[q^{-\frac{1}{2}}A] \Big).
\end{split}
\end{equation*}
We may compute the series expansion of the right-hand side by using
\[H_u[q^{\pm\frac{1}{2}}A]=\sum_{r\ge 0} u^r q^{\pm\frac{r}{2}} Q_{(r)}.\] 
Then if we write (4.2) as
\[e_{(\la_1,\la_2)}=(t^nq^{\la_1/2}+q^{-\la_1/2})(t^{n-1}q^{\la_2/2}+q^{-\la_2/2})\prod_{i=1}^{n-2}(t^i+1),\]
we have
\begin{multline*}
\sum_{\la_1,\la_2\in \mathbb{N}} u^{\la_1}v^{\la_2}\, e_{(\la_1,\la_2)} Q_{(\la_1)}Q_{(\la_2)}=
\prod_{i=1}^{n-2}(t^i+1)
\Big( t^{2n-1}H_u[q^{\frac{1}{2}}A]
H_v[q^{\frac{1}{2}}A] \\
+ t^{n-1} H_u[q^{-\frac{1}{2}}A]
H_v[q^{\frac{1}{2}}A]+ t^{n} H_u[q^{\frac{1}{2}}A]
H_v[q^{-\frac{1}{2}}A]
+  H_u[q^{-\frac{1}{2}}A]
H_v[q^{-\frac{1}{2}}A] \Big).
\end{multline*}
Thus we obtain
\begin{multline*}
\sum_{\la_1,\la_2\in \mathbb{N}} u^{\la_1}v^{\la_2}\, (E_\pi-e_{(\la_1,\la_2)}) Q_{(\la_1)}Q_{(\la_2)}
= (1-t) \prod_{i=1}^{n-2}(t^i+1)
\sum_{r,s\in \mathbb{N}} u^rv^s Q_{(r)}Q_{(s)}\\ \times
\Big(  \frac{uv/q}{1-tuv/q}(q^{\frac{1}{2}(r+s)}t^{2n}-q^{-\frac{1}{2}(r+s)})
+ t^{n-1} \frac{v}{u-v} (q^{\frac{1}{2}(-r+s)}
- q^{\frac{1}{2}(r-s)})\Big).
\end{multline*}
Identification of coefficients achieves the proof.
\end{proof}

\noindent \textit{Remark:} Starting from Theorem 6, the same method allows to obtain the polynomials $E_\pi (Q_{(\la_1)}\ldots Q_{(\la_r)})$ for $r\ge 2$.
However their generating function involves $2^{r}$ series and becomes quickly intricate.

\section{The root system $B_2$}

It is still unclear whether a similar method might be used for the root system $B_n$. However this result is easy to get for $n=2$, since Macdonald polynomials of type $B_2$ and $C_2$ are in bijective correspondence.

The set of positive roots of the root system $B_2$ is the union 
of $R_{1}=\{\ep_1+\ep_2, \ep_1-\ep_2\}$ and 
$R_{2}=\{\ep_1, \ep_2\}$. The fundamental weights of $B_2$ are $\varpi_1=\ep_1,\varpi_2=\frac{1}{2}(\ep_1+\ep_2)$. The dominant weights $\la\in P^+$ are vectors 
$\la=\la_1 \ep_1+\la_2 \ep_2$ where $\la_1,\la_2$ are all in $\mathbb{N}$ or in $\mathbb{N}/2$ and $\la_1\ge \la_2$. The weight $\varpi_2$ is minuscule and $\varpi_1$ is quasi-minuscule.

The dual root system $C_2$ has one minuscule weight $\tilde{\pi}=\ep_1$, with Weyl group orbit $\{\pm\ep_1,\pm\ep_2\}$. We have
$$\Phi_{\tilde{\pi}} = \frac {1-tx_1x_2} {1-x_1x_2}
\, \frac {1-tx_1/x_2} {1-x_1/x_2} \,\frac {1-tx_1} {1-x_1}.$$ 
The translation operator $T_{\tilde{\pi}}$ acts on $A$ by
$T_{\tilde{\pi}}f(x_1,x_2)=f(qx_1,x_2).$
The Macdonald operator $E_{\tilde{\pi}}$ writes as
\begin{multline*}
E_{\tilde{\pi}} f=\sum_{\sigma =\pm 1} 
\frac {1-tx_1^\sigma x_2} {1-x_1^\sigma x_2}
\, \frac {1-tx_1^\sigma /x_2} {1-x_1^\sigma /x_2} \,\frac {1-tx_1^\sigma} {1-x_1^\sigma}\, 
f(q^{\sigma}x_1,x_2)\\
+\sum_{\sigma =\pm 1} \frac {1-tx_1x_2^\sigma} {1-x_1x_2^\sigma}
\, \frac {1-tx_2^\sigma/x_1} {1-x_2^\sigma/x_1} \,\frac {1-tx_2^\sigma} {1-x_2^\sigma}\, 
f(x_1,q^{\sigma}x_2).
\end{multline*}
The Macdonald polynomial 
$P_{\la_1 \ep_1+\la_2 \ep_2}$ is defined, up to a constant, by
$$E_{\tilde{\pi}} P_{\la_1 \ep_1+\la_2 \ep_2}=(t^3q^{\la_1}+q^{-\la_1}+t^2q^{\la_2}+tq^{-\la_2}) \,P_{\la_1 \ep_1+\la_2 \ep_2}.$$

By comparison with the Macdonald operator of $C_2$, we easily obtain the following bijective correspondence between Macdonald polynomials of type $B_2$ and $C_2$. Defining $y_1=x_1x_2$ and $y_2=x_1/x_2$, we have
\begin{equation*}
\begin{split}
P^{(C)}_{\la_1\ep_1+\la_2\ep_2}(x;q,t,T)&=P^{(B)}_{\la_2\varpi_1+(\la_1-\la_2)\varpi_2}(y;q,T,t)\\
&=P^{(B)}_{(\la_1+\la_2)\ep_1/2+(\la_1-\la_2)\ep_2/2}(y;q,T,t).
\end{split}
\end{equation*}

The following results are obvious consequences. We normalize Macdonald polynomials attached to $B_2$ by
\[Q_{\la_1 \varpi_1+\la_2 \varpi_2}=\frac{(t;q)_{\la_1}}{(q;q)_{\la_1}} \, \frac{(t;q)_{\la_2}}{(q;q)_{\la_2}}\,\frac{(q^{\la_2}t^2;q)_{\la_1}}{(q^{\la_2+1}t;q)_{\la_1}}\,P_{\la_1 \varpi_1+\la_2 \varpi_2}.\]
The generating function of the polynomials $Q_{r\varpi_2}(y_1,y_2)$ is given by
\begin{equation*}
\sum_{r\in \mathbb{N}} u^r\, Q_{r\varpi_2}(x_1x_2,x_1/x_2)=\frac{(tux_1;q)_\infty} {(ux_1;q)_\infty}\
\frac{(tu/x_1;q)_\infty} {(u/x_1;q)_\infty}
\,\frac{(tux_2;q)_\infty} {(ux_2;q)_\infty}\
\frac{(tu/x_2;q)_\infty} {(u/x_2;q)_\infty}.
\end{equation*}

\begin{theo}
For any partition $\la=(\la_1,\la_2)$ we have
\[Q_{\la_1\varpi_2}\, Q_{\la_2\varpi_2}=\sum_{\begin{subarray}{c}(i,j)\in \mathbb{N}^2\\ 0 \le i+j \le \la_2 \end{subarray}} c_{ij}(\la_1,\la_2)\arrowvert_{n=2} \, Q_{(\la_1+\la_2-2j)\ep_1/2+(\la_1-\la_2+2i)\ep_2/2}.\]
Conversely
\[Q_{(\la_1+\la_2)\ep_1/2+(\la_1-\la_2)\ep_2/2}
=\sum_{\begin{subarray}{c}(i,j)\in \mathbb{N}^2\\ 0 \le i+j \le \la_2 \end{subarray}} C_{ij}(\la_1,\la_2)\arrowvert_{n=2} \, Q_{(\la_1+i-j)\varpi_2}\, Q_{(\la_2-i-j)\varpi_2}.\]
\end{theo}

The extension to any root system of rank 2 is an interesting problem. Of course the only remaining case is $R=G_2$, which might be investigated by using the Pieri formulas of van Diejen and Ito~\cite{D2}.

\section{Basic hypergeometric identities}

A referee has observed that the multivariate identity of Theorem 6 leads to basic $q$-hypergeometric identities by specialization of the indeterminates $x$ or $u$. In this section we present some of these corollaries, and explicit their connection with known results. 

We adopt the notation of~\cite{GR} and write
$${}_{r+1}\phi_r \left[\begin{matrix}
a_1,a_2,\dots,a_{r+1}\\
b_1,b_2,\dots,b_r \end{matrix};q,z \right]=
\sum_{i\ge 0}\frac{(a_1;q)_i\ldots(a_{r+1};q)_i}
{(b_1;q)_i\ldots(b_r;q)_i}\frac{z^i}{(q;q)_i}.$$
A first identity is obtained by the ``principal specialization'' of the $u$ indeterminates, i.e. $u_i=zt^{i-1}$ ($1\leq i\leq r$).
\begin{theo}
Let $x=(x_1,\ldots,x_n)$. We have the very-well-poised $q$-hypergeometric sum
\begin{multline*}
\sum_{\sigma\in (-1,+1)^n} 
\prod_{i=1}^n \frac{1-qx_{i}^{2\ta_i}} 
{1-x_{i}^{2\ta_i}}\,
\frac{1-q^rzx_i^{-\ta_i}}
{1-zx_i^{-\ta_i}}\,
\prod_{1\le i < j \le n}
\frac {1-qx_{i}^{\ta_i}x_{j}^{\ta_j}} 
{1-x_{i}^{\ta_i}x_{j}^{\ta_j}}=\\
q^{nr-\binom{r}{2}}\,(-q;q)_{n-r}\, 
\frac{(qz^2;q^2)_r}{(qz^2;q)_r}\,
{}_{2n+4}{\phi}_{2n+3} 
\left[\begin{matrix} z^2,qz,-qz,q^{-r},\{qzx_i,qz/x_i\}\\
z,-z,q^{r+1}z^2,\{zx_i,z/x_i\} \end{matrix};q,-q^{r-n} \right].
\end{multline*}
\end{theo}
\begin{proof}
With $u_i=zt^{i-1}$ ($1\leq i\leq r$), on the right-hand side of (8.3) only the subsets $I$ of the form $I=\emptyset$ and $I=\{1,2,\ldots,m\}$ with $1\le m \le r$ have nonvanishing contributions. Actually if there exist $j=i-1$ with $i \in I$ and $j\notin I$, we have $v_i=zt^{i-1}$ and $v_j=t^{-j}/z$ so that $v_iv_j=1$.

Thus the right-hand side of (8.3) writes as
\[
(-t;t)_{n-r} \, \sum_{m=0}^r 
t^{n(r-m)}\,c_m\,
\prod_{i=1}^n \frac{(tzx_i;t)_m}{(zx_i;t)_m}
\,\frac{(tz/x_i;t)_m}{(z/x_i;t)_m},\]
with
\begin{align*}
c_m &=
t^{-\binom{r-m}{2}}\,
\prod_{1\le i < j \le m}\frac{1-t^{i+j-2}z^2}{1-t^{i+j-1}z^2}\,
\prod_{m+1\le i < j \le r}\frac{1-t^{i+j}z^2}{1-t^{i+j-1}z^2}\,
\prod_{1\le i \le m < j \le r}\frac{1-t^{i-j-1}}{1-t^{i-j}}\\
&=t^{-\binom{r-m}{2}}\, 
\prod_{i=1}^m \frac{1-t^{2i-1}z^2}{1-t^{m+i-1}z^2}\,
\prod_{i=m+1}^r \frac{1-t^{r+i}z^2}{1-t^{2i}z^2}\,
\prod_{i=1}^m \frac{1-t^{i-r-1}}{1-t^{i-m-1}}\\
&=(-1)^mt^{\binom{m+1}{2}-\binom{r-m}{2}}\,\prod_{i=1}^r \frac {1-t^{r+i}z^2} {1-t^{2i}z^2}\,
\frac{1-t^{2m}z^2}{1-z^2}\,
\frac{(z^2;t)_m}{(t^{r+1}z^2;t)_m}\,
\frac{(t^{-r};t)_m}{(t;t)_m}.
\end{align*}
Hence the statement.
\end{proof}

In Theorem 8 we may perform a principal specialization of the $x$ indeterminates, i.e. $x_i=wq^{i-1}$ ($1\leq i\leq n$). This leads to the following transformation between a $_4\phi_3$ and a very-well-poised $_6\phi_5$.
\begin{theo} 
We have
\begin{align*}
&\frac{(qw^2;q^2)_n}{(w^2;q)_n}\,\frac{(q^{-r}w/z;q)_n}{ (w/z;q)_n}
{}_{4}\phi_3 \left[\begin{matrix}
w/z,q^rwz,w^2/q,q^{-n}\\
wz,q^{-r}w/z,q^nw^2 \end{matrix};q,-q^{n-r+1} \right]\\
&=q^{-\binom{r}{2}}\,(-q;q)_{n-r}\frac{(qz^2;q^2)_r}{(qz^2;q)_r}\,
{}_{6}\phi_5 \left[\begin{matrix}
z^2,qz,-qz,q^nwz,qz/w,q^{-r}\\
z,-z,wz,q^{1-n}z/w,q^{r+1}z^2 \end{matrix};q,-q^{r-n} \right].
\end{align*}
\end{theo}
\begin{proof}Due to
\[\prod_{i=1}^n\frac{(q^{i}wz;q)_m}{(q^{i-1}wz;q)_m}\frac{(q^{2-i}z/w;q)_m}{(q^{1-i}z/w;q)_m}=\frac{(q^n wz;q)_m}{(wz;q)_m}\frac{(q z/w;q)_m}{(q^{1-n}z/w;q)_m},\]
in Theorem 8 the right-hand side writes as 
\[(-q;q)_{n-r}\, q^{nr-\binom{r}{2}}\,
\frac{(qz^2;q^2)_r}{(qz^2;q)_r}\,
{}_{6}\phi_5 \left[\begin{matrix}
z^2,qz,-qz,q^nwz,qz/w,q^{-r}\\
z,-z,wz,q^{1-n}z/w,q^{r+1}z^2 \end{matrix};q,-q^{r-n} \right].\]
In the sum on the left-hand side, when $\sigma_i=+1$ for some $i$ only terms with $\sigma_{i+1}=+1$ have a non-zero contribution. Actually if $\sigma_{i+1}=-1$ the last factor $1-qx_{i}^{\ta_i}x_{i+1}^{\ta_{i+1}}$ vanishes. Therefore the non-zero terms are obtained for $\sigma_i=-1$ ($1\le i\le m$) and $\sigma_i=+1$ ($m+1\le i \le n$), with $0\le m\le n$. 

Therefore the left-hand side writes as
\begin{align*}
\sum_{m=0}^n
&\prod_{i=1}^m \frac{1-q^{3-2i}/w^2} 
{1-q^{2-2i}/w^2}\,
\frac{1-q^{r+i-1}wz}
{1-q^{i-1}wz}\,
\prod_{i=m+1}^n \frac{1-q^{2i-1}w^2} 
{1-q^{2i-2}w^2}\,
\frac{1-q^{r-i+1}z/w}
{1-q^{1-i}z/w}\\
\times
&\prod_{i=1}^m
\frac {1-q^{2-2i}/w^2} 
{1-q^{2-i-m}/w^2}
\prod_{i=1}^m
\frac {1-q^{n-i+1}} 
{1-q^{m-i+1}}
\prod_{i=m+1}^n
\frac {1-q^{i+n-1}w^2} 
{1-q^{2i-1}w^2}.
\end{align*}
This is easily transformed to
\begin{multline*}
\sum_{m=0}^n (-1)^n
q^{\binom{m+1}{2}+r(n-m)+\binom{n+1}{2}-\binom{n-m+1}{2}}\\
\times \frac{(w^2/q;q^2)_m} 
{(q^{m-1}w^2;q)_m}\,
\frac{(q^r wz;q)_m}
{(wz;q)_m}\,
\frac{(q^nw^2;q)_n} 
{(q^nw^2;q)_m}\, \frac{(w^2;q^2)_m} 
{(w^2;q^2)_n}\,
\frac{(q^{-r}w/z;q)_n} 
{(q^{-r}w/z;q)_m}\, \frac{(w/z;q)_m} 
{(w/z;q)_n}\,\frac{(q^{-n};q)_m}{(q;q)_m}\\
= q^{nr}
\frac{(q^nw^2;q)_n}  
{(w^2;q^2)_n}\,
\frac{(q^{-r}w/z;q)_n}  
{(w/z;q)_n}\,{}_{4}\phi_3 \left[\begin{matrix}
w/z,q^rwz,w^2/q,q^{-n}\\
wz,q^{-r}w/z,q^nw^2 \end{matrix};q,-q^{n-r+1} \right].
\end{multline*}
\end{proof}

Schlosser has given another proof by combining two identities of \cite{GR}, firstly Exercise 2.13 (ii) p.60, with $(a,b,c,d):=(w^2/q,w/z,q^rwz,q^{-n})$, then Equ. (2.11.1) p.53, with $(a,b,c,d,e,f):=(w^2q^{n-r-1},wq^{-1/2},wq^{n-r}/z,-wq^{-1/2},wzq^n,q^{-r})$.

An extension of Theorem 8 may be obtained by ``multiple principal specialization'', defined as follows. Given $s$ positive integers $\mathbf{k}=(k_1,\ldots,k_s)$ summing to $|\mathbf{k}|=r$, we consider the $s$ intervals $[\mathbf{k}_{l-1}+1,\mathbf{k}_l]$ where $\mathbf{k}_l$ denotes the partial sum $\sum_{j=1}^lk_j$. Given $r$ indeterminates $u=(u_1,\ldots,u_r)$, their multiple principal specialization is defined by
\[u_i=z_lt^{i-1-\mathbf{k}_{l-1}},\quad i\in [\mathbf{k}_{l-1}+1,\mathbf{k}_l]\qquad (1 \le l \le s).\]
References may be found in~\cite[Section 4.1]{LSW}. 
\begin{theo}
Let $x=(x_1,\ldots,x_n)$, $s$ positive integers $\mathbf{k}=(k_1,\ldots,k_s)$ and $z=(z_1,\ldots,z_s)$. We have the basic hypergeometric sum
\begin{multline*}
\sum_{\sigma\in (-1,+1)^n}
\prod_{i=1}^n \frac{1-qx_{i}^{2\ta_i}} 
{1-x_{i}^{2\ta_i}}\,
\prod_{l=1}^s\frac{1-q^{k_l}z_lx_i^{-\ta_i}}
{1-z_lx_i^{-\ta_i}}\,
\prod_{1\le i < j \le n}
\frac {1-qx_{i}^{\ta_i}x_{j}^{\ta_j}} 
{1-x_{i}^{\ta_i}x_{j}^{\ta_j}}\\
=q^{n|\mathbf{k}|-\binom{|\mathbf{k}|}{2}}\,(-q;q)_{n-|\mathbf{k}|}\, 
\prod_{1\le a< b\le s}\frac{(q^{k_b+1}z_az_b;q)_{k_a}}{(qz_az_b;q)_{k_a}}\,
\prod_{a=1}^s
\frac{(qz_a^2;q^2)_{k_a}}{(qz_a^2;q)_{k_a}}
\\\times
\Big(\sum_{\mathbf{m}\le \mathbf{k}}
(-1)^{|\mathbf{m}|}c_{\mathbf{m}}\,q^{(|\mathbf{k}|-n)|\mathbf{m}|}
\,\prod_{i=1}^n \prod_{l=1}^s \frac{(qx_iz_l;q)_{m_l}}{(x_iz_l;q)_{m_l}}
\,\frac{(qz_l/x_i;q)_{m_l}}{(z_l/x_i;q)_{m_l}}\Big),
\end{multline*}
with
\[
c_{\mathbf{m}}=
\prod_{1\le a<b\le s}\frac{q^{m_a}z_a-q^{m_b}z_b}{z_a-z_b}\,
\prod_{1\le a\le b\le s}\frac{1-q^{m_a+m_b}z_az_b}{1-z_az_b}\,
\prod_{a,b=1}^s\frac{(q^{-k_b}z_a/z_b;q)_{m_a}}{(qz_a/z_b;q)_{m_a}}\,
\frac{(z_az_b;q)_{m_a}}{(q^{1+k_b}z_az_b;q)_{m_a}}.
\]
\end{theo}
\begin{proof}
We write (8.3) with $u_i$ defined as above. By the same argument as in the proof of Theorem 8, on the right-hand side only the subsets $I$ of the form $I=\cup_{l=1}^s I_l$, with $I_l=\{\mathbf{k}_{l-1}+i, \, 1\le i \le m_l\}$ and $0\le m_l \le k_l$, have nonvanishing contributions. We denote $J_l=\{\mathbf{k}_{l-1}+i, \, m_l+1\le i \le k_l\}$ the complement of $I_l$ in $[\mathbf{k}_{l-1}+1,\mathbf{k}_l]$.

Then the right-hand side of (8.3) writes as
\[
(-t;t)_{n-|\mathbf{k}|} \, \sum_{\mathbf{m}\le \mathbf{k}} 
t^{n(|\mathbf{k}|-|\mathbf{m}|)}\,c_{\mathbf{m}}\,
\prod_{i=1}^n \prod_{l=1}^s \frac{(tx_iz_l;t)_{m_l}}{(x_iz_l;t)_{m_l}}
\,\frac{(tz_l/x_i;t)_{m_l}}{(z_l/x_i;t)_{m_l}}.\]
The coefficient $c_{\mathbf{m}}$ is the product of the following contributions. If $i\in I_a$ then either $j\in I_b$ for $b\ge a$, or $j\in J_b$ for $b> a$. Such cases contribute to
\begin{multline*}
\prod_{i=1}^{m_a} \prod_{1\le a< b\le s}\frac{1-t^{i-1}z_az_b}{1-t^{i+m_b-1}z_az_b} \, \frac{1-t^{i-k_b-1}z_a/z_b}{1-t^{i-m_b-1}z_a/z_b}\prod_{a=1}^s \frac{1-t^{2i-1}z_a^2}{1-t^{i+m_a-1}z_a^2}\,\frac{1-t^{i-k_a-1}}{1-t^{i-m_a-1}}.
\end{multline*}
If $i\in J_a$ then either $j\in I_b$ for $b>a$, or $j\in J_b$ for $b\ge a$. Such cases contribute to
\begin{multline*}
\prod_{i=m_a+1}^{k_a} \prod_{1\le a< b\le s}
\frac{1-t^{-i}z_b/z_a}{1-t^{m_b-i}z_b/z_a}\,t^{m_b-k_b}\,
\frac{1-t^{k_b+i}z_az_b}{1-t^{m_b+i}z_az_b}\,
\prod_{a=1}^s t^{-\binom{k_a-m_a}{2}}\,
\frac{1-t^{k_a+i}z_a^2}{1-t^{2i}z_a^2}.
\end{multline*}
The statement follows by easy, but tedious, transformations.
\end{proof}

The multidimensional series at the right-hand side involves the factor
\[\prod_{1\le a<b\le s}\frac{q^{m_a}z_a-q^{m_b}z_b}{z_a-z_b}\,
\prod_{1\le a\le b\le s}\frac{1-q^{m_a+m_b}z_az_b}{1-z_az_b}.\]
Such series are called $C_n$ basic hypergeometric series. There is a rich litterature in this domain. Many references may be found in the bibliography of~\cite{Ro2,S}. 

A (messy) multiple principal specialization of the $x$ indeterminates might be performed on Theorem 10. Here we shall only consider a simple principal specialization $x_i=wq^{i-1}$ ($1\leq i\leq n$).
A proof strictly parallel to Theorem 9 leads to the following basic hypergeometric identity.
\begin{theo}
We have
\begin{multline*}
\frac{(qw^2;q^2)_n}{(w^2;q)_n}\,
\prod_{l=1}^s \frac{(q^{-k_l}w/z_l;q)_n}{ (w/z_l;q)_n}\,
{}_{2s+2}\phi_{2s+1} \left[\begin{matrix}
w^2/q,q^{-n},\{w/z_l,q^{k_l}wz_l\}\\
q^nw^2,\{wz_l,q^{-k_l}w/z_l\} \end{matrix};q,-q^{n-|\mathbf{k}|+1} \right]\\
=q^{-\binom{|\mathbf{k}|}{2}}\,(-q;q)_{n-|\mathbf{k}|}\, 
\prod_{1\le a< b\le s}\frac{(q^{k_b+1}z_az_b;q)_{k_a}}{(qz_az_b;q)_{k_a}}\,\prod_{a=1}^s
\frac{(qz_a^2;q^2)_{k_a}}{(qz_a^2;q)_{k_a}}\\\times
\Big(\sum_{\mathbf{m}\le \mathbf{k}}
(-1)^{|\mathbf{m}|}c_{\mathbf{m}}\,q^{(|\mathbf{k}|-n)|\mathbf{m}|}
\, \prod_{l=1}^s \frac{(q^nwz_l;q)_{m_l}}{(wz_l;q)_{m_l}}
\,\frac{(qz_l/w;q)_{m_l}}{(q^{1-n}z_l/w;q)_{m_l}}\Big),
\end{multline*}
with $c_{\mathbf{m}}$ as in Theorem 10.
\end{theo}

Schlosser has pointed out that Theorem 10 is already known, and corresponds to a particular case of an identity due to Rosengren~\cite[Corollary 4.4, p.342]{Ro}. 

Indeed in this result, if we substitute $l=(l_1,\ldots,l_n) \to \mathbf{k}=(k_1,\ldots,k_s)$, $y=(y_1,\ldots,y_n) \to \mathbf{m}=(m_1,\ldots,m_s)$, $c=(c_1,\ldots,c_p) \to x=(x_1,\ldots,x_n)$, and specialize $b=-d=q^{\frac{1}{2}}$ and $(m_1,\ldots,m_p)=(1,\ldots,1)$, we recover Theorem 10. 

Incidentally the identity of~\cite{Ro} has a symmetrical form since it writes the left-hand side of Theorem 10 as the following $C_n$ basic hypergeometric series
\begin{multline*}
\prod_{i=1}^n \frac{1-qx_{i}^2} 
{1-x_{i}^2}\,
\prod_{l=1}^s\frac{1-q^{-k_l}x_i/z_l}
{1-x_i/z_l}\,
\prod_{1\le i < j \le n}
\frac {1-qx_{i}x_{j}} 
{1-x_{i}x_{j}} \\
\times\sum_{\sigma\in (0,1)^n}
(-1)^{|\sigma|}d_{\sigma}\,q^{(n-|\mathbf{k}|+1)|\sigma|}
\,\prod_{i=1}^n \prod_{l=1}^s \frac{(q^{k_l}x_iz_l;q)_{\sigma_i}}{(x_iz_l;q)_{\sigma_i}}
\,\frac{(x_i/z_l;q)_{\sigma_i}}{(q^{-k_l}x_i/z_l;q)_{\sigma_i}},
\end{multline*}
with
\begin{multline*}
d_{\sigma}=\prod_{1\le a<b\le n}\frac{x_aq^{\sigma_a}-x_bq^{\sigma_b}}{x_a-x_b}\,
\prod_{1\le a\le b\le n}\frac{1-x_ax_bq^{\sigma_a+\sigma_b-1}}{1-x_ax_b/q} \\\times
\prod_{a=1}^n \frac{(x_a^2/q;q^2)_{\sigma_a}}{(qx_a^2;q^2)_{\sigma_a}} 
\prod_{a,b=1}^n\frac{(x_ax_b/q;q)_{\sigma_a}}{(qx_a/x_b;q)_{\sigma_a}}\,
\frac{(x_a/qx_b;q)_{\sigma_a}}{(qx_ax_b;q)_{\sigma_a}}.
\end{multline*}

The connection of our multivariate identity with~\cite{Ro} and $C_n$ basic hypergeometric series opens several interesting questions. 

Firstly, since~\cite{Ro} has an $A_n$ counterpart~\cite{Ro1}, is it also the case for Theorem 6  ? More generally does such an identity exist for any root system ? Secondly, since the result of~\cite{Ro} depends on indeterminates $(m_1,\ldots,m_p)$ and has a symmetrical form, is such a generalization also possible for Theorem 6 ? Finally would it be even possible to generalize Theorem 6 in the framework of elliptic functions~\cite{LSW,Ro2} ?

\end{document}